\documentclass[11pt]{amsart}
\usepackage{amsfonts, amstext, amsmath, amsthm, amscd, amssymb}
\usepackage[latin1]{inputenc}
\usepackage[T1]{fontenc}
\usepackage{layout}

\usepackage[english,german]{babel}
\usepackage{enumerate}

\usepackage{type1cm}

\renewcommand{\leq}{\leqslant}

\newcommand{\Z}{\mathbb{Z}}

\newcommand{\R}{\mathbb{R}}
\newcommand{\C}{\mathbb{C}}

\newcommand{\Tr}{\mathop{\mathrm{Tr}}}

\newtheorem{thm}{Theorem}
\newtheorem{lem}{Lemma}
\newtheorem{prop}{Proposition}
\newtheorem{cor}{Corollary}
\newtheorem{defi}{Definition}
\newtheorem{rem}{Remark}

\title{Co-Euler structures on bordisms}
\author{Osmar MALDONADO MOLINA}
  \address{Department of Mathematics
University of Vienna,
Oskar-Morgenstern-Platz 1,
 Room 03.131
 A-1090 Vienna}
 \email{osmar.maldonado@univie.ac.at}

\begin{document}
\selectlanguage{english}
\thispagestyle{empty}

\begin{abstract}
Co-Euler structures were studied by Burghelea and Haller on closed manifolds as dual objects to Euler structures.
We extend the notion of co-Euler structures to the situation of compact manifolds with boundary.
As an application, by studying their variation with respect to smooth changes of the Riemannian metric,
co-Euler structures conveniently provide correction terms that can be taken into account when considering the complex-valued analytic torsion on bordisms as a Riemannian invariant.
\end{abstract}

\maketitle

\section{Introduction}
In this paper,
$M$ is considered to be a compact connected non-necessarily oriented $m$-dimensional manifold with Riemannian metric $g$, and boundary $\partial M$ that inherits its Riemannian metric from that of $M$. Moreover, we assume $\partial M$ to be
the disjoint union of two closed (non-necessarily connected) submanifolds $\partial_{ +}M$ and $\partial_{-}M$.
We write
\begin{equation}
\label{bordism}
\mathbb{M}:=(M,\partial_{+}M,\partial_{-}M)\quad\text{and}\quad\mathbb{M}^{\prime}:=(M,\partial_{-}M,\partial_{+}M)
\end{equation}
to indicate that $M$ is considered as a bordism from $\partial_{+}M$ to $\partial_{-}M$, and $\mathbb{M}^{\prime}$ for
its \emph{dual} bordism, i.e., $M$ seen as the bordism from $\partial_{-}M$ to $\partial_{+}M$.

The concept of Euler structures was first introduced by Turaev in \cite{Turaev90}, see also \cite{Farber-Turaev}, for manifolds $M$ with vanishing Euler--Poincar\'e
characteristics $\chi(M)$ to conveniently remove ambiguities in the definition of the Reidemeister torsion. The set of Euler structures $\mathbf{\mathfrak{Eul}}(M;\C)$
is an affine space over the homology group $H_1(M;\C)$ in the sense that $H_1(M;\C)$ acts
freely and transitively on $\mathbf{\mathfrak{Eul}}(M;\C)$. Then, Euler structures were studied
on manifolds with arbitrary Euler characteristics at the expense of introducing a base point $x_{0}\in M$, see \cite{Burghelea_Removing metric anomalies from Ray--Singer torsion}.

Co-Euler structures can be considered as dual objects to Euler structures and were introduced by Burghelea and Haller in \cite{Burghelea-Haller06a} and \cite{Burghelea-Haller06b} and then used
in \cite{Burghelea-Haller} to study (variational formulas of) the complex-valued analytic torsion given on closed manifolds.
To have the ideas set up, let us recall in the situation of closed manifolds what Co-Euler structures are. Assume that $M$ is closed, connected and that its Euler characteristics $\chi(M)=0$. If $\Theta_{M}$ indicates the orientation bundle of $M$ and $\Theta_{M}^{\C}$ its complexification, the set of co-Euler
structures $ \mathbf{\mathfrak{Eul}}^{*}(M;\C)$ is an affine version of cohomology groups $H^{m-1}(M;\Theta^{\C}_M)$.
A co-Euler structure is an equivalence class of pairs $(g,\alpha)$, where $g$
is a Riemannian metric on $M$, and $\alpha\in\Omega^{m-1}(M;\Theta_M^{\C})$ is a $(m-1)$-smooth differential form  over $M$
with $d\alpha=\mathbf{e}(g)$ where $\mathbf{e}(g)\in\Omega^{m}(M;\Theta_M^{\C})$ is the Euler form of $g$. Two
such pairs $(g_1,\alpha_1)$ and $(g_2,\alpha_2)$ are equivalent if and only if
$\alpha_2-\alpha_1=\widetilde{\mathbf{e}}(g_1,g_2)$ where
$\widetilde{\mathbf{e}}(g_1,g_2)\in\Omega^{m-1}(M;\Theta_M^{\C})/d\Omega^{m-2}(M;\Theta^{\C}_M)$
denotes the Chern--Simons form. 
By construction co-Euler structures were well
suited to remove the metric ambiguities of the analytic torsion on closed manifolds and finally
provide a topological invariant, referred as the \textit{modified} Ray--Singer torsion, see \cite{Burghelea-Haller06a}.

In Section \ref{section_Co_Euler_structures}, we define co-Euler structures on a bordism $\mathbb{M}$.
As in case for a closed manifold, we start with the case where the relative Euler characteristics $\chi(M,\partial_{+}M)$ (or equivalently $\chi(M,\partial_{-}M)$) vanishes. In this situation
the space $\mathbf{\mathfrak{Eul}}_{}^{*}(\mathbb{M};\C)$ of co-Euler structures on $\mathbb{M}$
can be seen as an affine space over the relative cohomology group
$H^{m-1}(M,\partial M;\C)$ and depends on the choice of a base point $x_0\in M$ if the Euler characteristics $\chi(M;\partial_{-}M)\neq 0$.

A co-Euler structure on $\mathbb{M}$ is an equivalence class represented by
couples $(\underline{\alpha},g)$, where  $\underline{\alpha}=(\alpha,\alpha_{\partial})$ is a \textit{relative form} in the relative cochain complex $\Omega^{m-1}(M,\partial M;\Theta_{M}^{\C})$, i.e.,
a pair of differential forms $\alpha\in\Omega^{m-1}(M;\Theta_{M}^{\C})$ and $\alpha_{\partial}\in\Omega^{m-2}(\partial M;\Theta_{\partial M}^{\C})$
with $\mathbf{d}\underline{\alpha}=\underline{\mathbf{e}}(g)$ where $\mathbf{d}$ is an appropriate differential on the relative complex,
and $\underline{\mathbf{e}}(\mathbb{M},g)\in\Omega^{m}(M,\partial M;\Theta_M^{\C})$ is a relative Euler form associated to $(\mathbb{M},g)$. Two
such pairs $(g_1,\underline{\alpha}_1)$ and $(g_2,\underline{\alpha}_2)$ are equivalent if and only if
$\underline{\alpha}_2-\underline{\alpha}_1=\underline{\widetilde{\textbf{e}}}(\mathbb{M},g_1,g_2)$ where
$\underline{\widetilde{\textbf{e}}}(\mathbb{M},g_1,g_2)\in\Omega^{m-1}(M,\partial M;\Theta_M^{\C})$ modulo exact relative forms $\mathbf{d}\Omega^{m-2}(M;\partial M;\Theta^{\C}_M)$.

The relative Euler form and the relative Chern--Simons' forms on $M$ (and on $\partial M$) that we use are based on those worked out by Br\"uning and Ma in \cite{Bruening-Ma}, which appear in the anomaly formulas for the Ray--Singer metric, see \cite[Theorem 0.1]{Bruening-Ma}
and  \cite[Theorem 3.4]{Bruening-Ma2}, and also in the anomaly formulas for the complex-valued
Ray--Singer torsion, see \cite[Theorem 2]{Maldonado}. For the reader's convenience, we explain in the Appendix how these characteristic forms are constructed.

Moreover, we explain how co-Euler structures
on $\mathbb{M}$ are in a one-to-one correspondence with a co-Euler structure
on its dual bordism ${\mathbb{M}^{\prime}}$, by means of a so-called \textit{flip} map $\nu^{*}$, compatible with Poincar\'e
duality and affine over involution in relative cohomology.

In Proposition \ref{Lemma_variation_co-Euler_structures_on_man_with_boundary_without_base_point}, we derive the
infinitesimal variation of representatives of co-Euler structures with respect to smooth changes in the Riemannian metric, which then is used
in Section \ref{Section_Burghelea--Haller analytic torsion on bordisms} to encode
the variation of the complex-valued Ray--Singer torsion.

Then, more generally, we treat the case $\chi(M,\partial_{\pm}M)\not=0$, by considering a base point $x_{0}\in M$ and we define the space of base-pointed co-Euler structures
denoted by  $\mathbf{\mathfrak{Eul}}_{x_{0}}^{*}(\mathbb{M};\C)$.
We obtain
Proposition \ref{Proposition_regularization_function_S_both_boundaries} where we study their infinitesimal variation,
by using a regularization procedure for relative forms having
a \textit{singularity} in the interior of $M$ only.

In Section \ref{Section_Euler_structure_and_Co-Euler_Structures}, we recall 
the space of Euler structures on manifolds with boundary defined by Turaev in \cite{Turaev90}. We use a relative Mathai--Quillen form, 
to study Poincar\'e--Lefschetz duality in terms of a canonic isomorphism relating Euler and co-Euler structures in this setting. The relative
Mathai--Quillen form as presented here can be used to compare the complex-valued analytic torsion and the Milnor torsion without the need of (co)-Euler structures.

In Section \ref{Section_Burghelea--Haller analytic torsion on bordisms}, we define
a modified version for the complex-valued Ray--Singer torsion on
compact bordisms,  by conveniently adding certain correction terms. These correction terms,
expressed in terms of co-Euler structures, are incorporated to cancel out the variation of the
complex-valued Ray--Singer torsion with respect to smooth variations of the Riemannian metric and bilinear structures, given
in \cite[Theorem 2]{Maldonado}. In analogy with the situation on closed manifolds,
the modified complex-valued analytic torsion depends on the flat connection, the homotopy class
of the bilinear form and the co-Euler structure only. Finally, by means of the flip map $\nu^{*}$, we show naturality of the the modified torsion with respect to Poincar\'e duality.


\section*{Acknowledgements}
I would like to thank Stefan Haller for his comments on this paper and in particular for pointing out the construction and the use of the relative Mathai--Quillen form in Section \ref{Section_Euler_structure_and_Co-Euler_Structures}. Also I thank the University of Vienna for its support.

\section{Generalities and some conventions}
 \label{section_background}
 Consider the bordism $\mathbb{M}$ in (\ref{bordism}) and denote by $i:\partial M\hookrightarrow M$ the canonical embedding.
Let $\Theta_{M}\rightarrow M$ be  the orientation bundle of $TM$, considered as
the real line bundle associated to the frame bundle of $TM$, via the homomorphism
$\mathsf{sign} \det : GL_{m}(\R)\rightarrow O(1) \hookrightarrow GL_1(\R)$.
Since the structure group $O(1)=\{-1,+1\}$ is discrete, $\Theta_{M}$
is endowed with a canonical flat connection and a canonical fiber-wise metric which is parallel.
 We denote by $\Theta_{M}^{\C}$ the complexification of
 $\Theta_{M}$. As usual, $\Omega^{q}(M)$ is the vector space of smooth differential $q$-forms on $M$ so that
 $\Omega(M):=\oplus_{q}^{m}\Omega^{q}(M)$ is the de-Rham cochain complex of differential forms with de-Rham differential $d$. Thus,
 $\Omega(M;\Theta_{M})$ is the de-Rham cochain complex of $\Theta_{M}$-valued differential forms with induced differential still denoted by $d$.
 Analogously, we denote by $\Theta_{\partial M}$ the orientation bundle of $T \partial M$ and, as real line bundles over $\partial M$, we identify  $\Theta_{M}|_{\partial M}:=i^{*}\Theta_{M}$
 with $\Theta_{\partial M}$ by using the \textit{outward normal first convention}.
 The corresponding Levi--Civit\`a connections on $T M$ and $T \partial M$ are denoted by $\nabla$ and by $\nabla^{\partial}$ respectively.
 Recall the Hodge $\star$-operator
 $
 \star_{q}:=\star_{g,q}:\Omega^{q}(M)\rightarrow\Omega^{m-q}(M;\Theta_{M}),
 $
 i.e., the linear isomorphism
 defined by
 $
 \alpha\wedge\star\alpha^{\prime}=\langle\alpha,\alpha^{\prime}\rangle_{g}\mathsf{vol}_{g}(M),
 $
 for $\alpha,\alpha^{\prime}\in\Omega^{q}(M)$ and $0 \leq q\leq m$, where $\mathsf{vol}_{g}(M)\in\Omega^{m}(M;\Theta_{M})$ is the volume form of $M$.

 Recall that the \textit{relative cohomology} group $H^{q}(M,\partial M;\Theta_{M}^{\C})$ in degree $q$, can be computed, see \cite{Bott-Tu},
 by means of the $\Z$-graded differential cochain complex
\begin{equation}
 \label{equation_relative_forms_without_base_point}
 \Omega(M,\partial M;\Theta_{M}^{\C}):=\oplus_{q=0}^{m}\Omega^{q}(M,\partial M;\Theta_{M}^{\C})\end{equation}
 where $$\Omega^{q}(M,\partial M;\Theta_{M}^{\C}):=
\Omega^{q}(M;\Theta_{M}^{\C})\oplus\Omega^{q-1}(\partial M;\Theta_{M}^{\C})$$
 with $\Omega^{-1}(\partial M;\Theta_{M}^{\C}):=0.$ The space (\ref{equation_relative_forms_without_base_point}) will be referred as the space of \textit{relative} differential forms.
 The differential map in (\ref{equation_relative_forms_without_base_point}) is defined by
 \begin{equation}
 \label{equation_differential_in_cone}
 \begin{array}{rrcl}
 \mathbf{d}:&\Omega^{q}(M,\partial M;\Theta_{M}^{\C})&\rightarrow&\Omega^{q+1}(M,\partial M;\Theta_{M}^{\C})\\
 &(\alpha,\alpha_{\partial})&\mapsto & (d \alpha,i^{*}\alpha-d^{\partial}\alpha_{\partial}),
 \end{array}
 \end{equation}
 where $i:\partial M\hookrightarrow M$ and $d^{\partial}$ is the de-Rham differential at the boundary.
 Note that $\Omega(M,\partial M;\Theta_{M}^{\C})$ can be considered as a $\Omega(M)$-module by setting
 $$
 (\alpha,\alpha_{\partial})\wedge w:=(\alpha\wedge w,\alpha_{\partial}\wedge i^{*} w),\quad\text{for }w\in\Omega(M).
 $$
 For simplicity, we denote relative forms by $$\underline{\alpha}:=(\alpha,\alpha_{\partial})\in\Omega(M,\partial M;\Theta_{M}^{\C}).$$ Then, we have the graded Leibinz formula
 \begin{equation}
 \label{equation_Leibnit_rule_relative_forms}
 \mathbf{d}(\underline{\alpha}\wedge w)=(\mathbf{d}\underline{\alpha})\wedge w+(-1)^{q}\underline{\alpha}\wedge d w,
 \end{equation}
 which holds for each $\underline{\alpha}\in\Omega^{q}(M,\partial M;\Theta_{M}^{\C})$ and $w\in\Omega(M)$.

 Furthermore,
 for $\underline{\alpha}\in\Omega^{q}(M,\partial M,\Theta_{M})$
 and $w\in\Omega^{m-q}(M)$, one has the pairing
 \begin{equation}
 \label{pairing}
 \int_{(M,\partial
 M)}\underline{\alpha}\wedge w:=\int_{(M,\partial
 M)}(\alpha,\alpha_{\partial})\wedge w:=
 \int_{M}\alpha \wedge w-\int_{\partial M}\alpha_{\partial}\wedge
 i^{*} w,
 \end{equation}
 which induces a non-degenerate pairing $\langle\cdot,\cdot\rangle$ in cohomology:
  \begin{equation}
 \label{pairing_in_cohomology}
\begin{array}{lrcl}
\langle\cdot,\cdot\rangle:& H^{*}(M,\partial
M,\Theta_{M}^{\C})\times H^{m-*}(M;\C)&\rightarrow&\C \\
&\langle[(\alpha,\alpha_{\partial})],[w]\rangle&\mapsto&\int_{(M,\partial M)}(\alpha,\alpha_{\partial})\wedge w.
\end{array}
 \end{equation}
 If in addition $M$ is connected, then non-degeneracy of $\langle\cdot,\cdot\rangle$ implies that
  \begin{equation}
 \label{equation_cohomology_top_degre_no_base_point}
 H^{m}\left(M,\partial M;\Theta_{M}^{\C}\right)\cong H^{0}(M;\C)^{\prime}\cong\C.
 \end{equation}

We will be also be interested in spaces with a base point. For $x_{0}\in M\backslash \partial M$ a base point in the interior of $M$, denote by
$
\dot{M}:=M\backslash\{x_{0}\}
$. Consider $$\Omega^{q}(\dot{M},\partial M;\Theta_{M}^{\C}):=\Omega^{q}(\dot{M};\Theta_{M}^{\C})\oplus\Omega^{q-1}(\partial M;\Theta_{M}^{\C})$$ so that
\begin{equation}
 \label{equation_relative_forms_with_base_point}
      \Omega(\dot{M},\partial M;\Theta_{M}^{\C}):=\oplus_{q=0}^{m}\Omega^{q}(\dot{M},\partial M;\Theta_{M}^{\C})
\end{equation}
endowed with the same differential map $\mathbf{d}$
as in (\ref{equation_differential_in_cone}), is also a $\Z$-graded complex.
In analogy with (\ref{equation_cohomology_top_degre_no_base_point}), if $M$ is connected, then it is not difficult to show, see \cite{Maldonado_thesis}, that
\begin{equation}
 \label{Lemma_top_degree_relative_cohomology_with_base_point_vanishes}
\begin{array}{rcl}
H^{m}(\dot{M},\partial M;\Theta_{M}^{\C})&\cong & 0 \\
H^{m-1}\left(M,\partial M;\Theta_{M}^{\C}\right)&\cong&
          H^{m-1}(\dot{M},\partial M;\Theta_{M}^{\C}).\\
\end{array}
\end{equation}

\section{Co-Euler structures}
\label{section_Co_Euler_structures}
  In order to construct co-Euler structures on a bordism $\mathbb{M}$,
  we first need to introduce certain characteristic forms and secondary characteristic forms on the manifold and on its boundary. These characteristic forms are essentially a \textit{modified} version of those already considered
  Br\"uning and Ma in \cite{Bruening-Ma} when studying the (variation of) Ray--Singer analytic torsion on manifolds with boundary. More precisely, the forms we need on $M$
  are the Euler form $\mathbf{e}(M,g)\in\Omega^{m}(M;\Theta_{M}^{\C})$ associated to the metric $g$,
  and secondary forms of Chern--Simons type
  $\mathbf{\widetilde{e}}(M,g,g^{\prime})\in\Omega^{m-1}(M;\Theta_{M}^{\C})$ associated to two (smoothly connected) Riemannian metrics
  $g$ and $g^{\prime}$. The characteristic form on $\partial M$ that we need is defined in \cite[expression (1.17), page 775]{Bruening-Ma} denoted by $\mathbf{e_{b}}(\partial M,g)$ 
  and the secondary (Chern--Simons) form is that in \cite[expression (1.45), page 780]{Bruening-Ma} and denoted $\mathbf{\widetilde{e}_{b}}(\partial M,g,g^{\prime})\in\Omega^{m-2}(\partial M;\Theta_{M}^{\C})$. The forms
  $\mathbf{e_{b}}(\partial M,g)$ and $\mathbf{\widetilde{e}_{b}}(\partial M,g,g^{\prime})$ were constructed by Br\"uning and Ma with respect to an inward pointing (unit vector) field along the whole boundary $\partial  M$. Here, we want to distinguish the roles of $\partial_{+}M$ and $\partial_{-}M$.
  We denote by $\varsigma_{\mathsf{in}}$ the unit \textit{inward pointing} normal vector field on the boundary, and by
 $\varsigma_{\mathsf{out}}:=-\varsigma_{\mathsf{in}}$ the unit \textit{outward pointing} normal vector field on the boundary.
  Then, we consider the following vector field
 \begin{equation}
 \label{definitions_out_in_unit_vector_fields_at_boundary_plus}
 \varsigma
				  :=\left\{\begin{array}{lr}\varsigma_{\mathsf{in}}&\text{on }
				  \in\partial_{+}M \\
				  \varsigma_{\mathsf{out}}&\text{on }\in\partial_{-}M\\
                        \end{array}
             \right.
 \end{equation}
 which is inward pointing along $\partial_{+}M$ and outward pointing along $\partial_{-}M$.
 Then, we use the vector field $\varsigma$ given in (\ref{definitions_out_in_unit_vector_fields_at_boundary_plus}) to specify a characteristic form $$\mathbf{e}_{\partial}(\partial_{+}M,\partial_{-}M,g)\in\Omega^{m-2}(\partial M;\Theta_{M}^{\C}),$$ a slightly modified version of $\mathbf{e_{b}}(\partial M,g)$, and
  a secondary characteristic form $$\left.\widetilde{\mathbf{e}}_{\partial}\right.\left(\partial_{+}M, \partial_{-}M,g,g^{\prime}\right)\in\Omega^{m-2}(\partial M,\Theta_{M}^{\C}),$$ a slightly modified version of $\mathbf{\widetilde{e}_{b}}(\partial M,g,g^{\prime})$, according to the vector field $\varsigma$. For further details, the reader is strongly referred at this point to the Appendix.

\begin{defi}
\label{definition_relative_Euler_form}
Let $\mathbb{M}$ be a Riemannian bordism.
Consider the forms from Definition \ref{definitions_of_modified_e_partial_B_forms_on_bordisms} in the Appendix. The \textbf{relative Euler form} is
$$
  \begin{array}{c}
  \underline{\mathbf{e}}(\mathbb{M},g)
 :=(\mathbf{e}(M,g),\mathbf{e}_{\partial}(\partial_{+}M,\partial_{-}M,g))
 \in\Omega^{m}(M,\partial M;\Theta_{M}^{\C}).
  \end{array}
$$
\end{defi}

The relative Euler form $\underline{\mathbf{e}}(\mathbb{M},g)$ is \textit{closed} in
$\Omega(M,\partial M;\Theta_{M}^{\C})$, because of dimensional reasons. From formula (\ref{equation_relations_chern_simons_relative_euler_forms_2}) in Lemma \ref{basic_properties_of_chern_simon_forms} below,
it follows that its cohomology class
\begin{equation}
 \label{equation_relations_chern_simons_relative_euler_forms_1}
\left[\underline{\mathbf{e}}(\mathbb{M})\right]:=\left[ \underline{\mathbf{e}}(\mathbb{M},g) \right]
\end{equation}
is independent of $g$.

\begin{defi}
\label{definition_secondary_relative_Euler_form}
The  \textbf{secondary relative Euler form} on $\mathbb{M}$ associated to the Riemannian metrics $g_{0}$ and $g_{\tau}$ is the relative form
$$\underline{\mathbf{\widetilde{e}}}(\mathbb{M},g_{0},g_{\tau})\in\Omega^{m-1}(M,\partial M;\Theta_{M}^{\C})$$ given by
\begin{equation}
\label{definition_Chern_Simons_relative_form_on_bordism}
\begin{array}{c}
\underline{\mathbf{\widetilde{e}}}(\mathbb{M},g_{0},g_{\tau})
:=\left(\widetilde{\mathbf{e}}\left(M,g_{0},g_{\tau}\right),
-\left.\widetilde{\mathbf{e}}_{\partial}\right.\left(\partial_{+}M,\partial_{-}M,g_{0},g_{\tau}\right)\right)
\end{array}
\end{equation}
where $\widetilde{\mathbf{e}}\left(M,g_{0},g_{\tau}\right)$ and
$\left.\widetilde{\mathbf{e}}_{\partial}\right.\left(\partial_{+}M,\partial_{-}M,
   g_{0},g_{\tau}\right)$ are the Chern--Simons forms given in Definition \ref{definition_of_Chern_Simons_secondary_classes_1} in the Appendix.
\end{defi}

\begin{lem}(Br\"uning--Ma)
\label{basic_properties_of_chern_simon_forms}
Let $\underline{\mathbf{\widetilde{e}}}(\mathbb{M},g_{0},g_{1})$ be the secondary relative Euler form in (\ref{definition_Chern_Simons_relative_form_on_bordism}) associated
to a couple of Riemannian metrics $g_{0}$, $g_{1}$ in $M$. If $\{g_{s}\}$ is a smooth path of Riemannian metrics connecting $g_{0}$ to $g_{1}$, then the formula
\begin{equation}
 \label{equation_relations_chern_simons_relative_euler_forms_2}
\mathbf{d}
\underline{\mathbf{\widetilde{e}}}(\mathbb{M},g_{0},g_{1})
=\underline{\mathbf{e}}(\mathbb{M},g_{1})-\underline{\mathbf{e}}(\mathbb{M},g_{1})
\end{equation}
holds. The secondary relative Euler form $\underline{\mathbf{\widetilde{e}}}(\mathbb{M},g_{0},g_{1})$ does depend on the path of metrics,
but only up to exact forms,
so that, it
defines a secondary
relative Euler class in the sense of Chern--Simons. Moreoever, up to exact forms in relative cohomology, the relations
\begin{equation}
 \label{equation_relations_chern_simons_relative_euler_forms_3}
 \begin{array}{l}
  \underline{\mathbf{\widetilde{e}}}(\mathbb{M},g_{0},g_{\tau})=-
  \underline{\mathbf{\widetilde{e}}}(\mathbb{M},g_{\tau},g_{0})\\
  \underline{\mathbf{\widetilde{e}}}(\mathbb{M},g_{0},g_{\tau})=
  \underline{\mathbf{\widetilde{e}}}(\mathbb{M},g_{0},g_{s}) +
  \underline{\mathbf{\widetilde{e}}}(\mathbb{M},g_{s},g_{\tau})
  \end{array}
\end{equation}
hold.
\end{lem}

\begin{proof}
 Since $\partial_{+}M$ and $\partial_{+}M$ are disjoint closed submanifolds, the statements above are exactly
 \cite[Theorem 1.9]{Bruening-Ma}.
 The identities in (\ref{equation_relations_chern_simons_relative_euler_forms_3}) follow  from the definition of
 $\underline{\mathbf{\widetilde{e}}}(\mathbb{M},g_{0},g_{\tau})$ in Definition
 \ref{definition_of_Chern_Simons_secondary_classes_1} in the Appendix.
\end{proof}

\subsection{Co-Euler structures without base point}
\label{Section_co-Euler Structures without base point}
We extend the notion of co-Euler structures in \cite{Burghelea-Haller} to the case of  bordisms $\mathbb{M}$.

\begin{lem}
\label{definition_equivalence_relation_without_base_point}
Recall Definitions \ref{definition_relative_Euler_form}, \ref{definition_of_Chern_Simons_secondary_classes_1}
  together with the
pairing $\langle\cdot,\cdot\rangle$ from (\ref{pairing}). Assume $M$ is connected. Let
$\underline{\mathbf{e}}(\mathbb{M},g)$ be the relative form given in Definition \ref{definition_relative_Euler_form}.
We assume first that the relative Euler Characteristics $\chi(M,\partial_{+}M)=0$.
Then the set
 \begin{equation}
\label{definition_Set_E_bordism}
\mathbf{E}^{*}(\mathbb{M};\C):=\left\{(g,\underline{\alpha})\left|
\begin{array}{rcl}
\underline{\alpha}&\in&\Omega^{m-1}(M,\partial M;\Theta_{M}^{\C})\\
\mathbf{d}\underline{\alpha}&=&\underline{\mathbf{e}}(\mathbb{M},g)\\
\end{array}
\right.
\right\}
\end{equation}
is not empty, so that we can define a relation in the space (\ref{definition_Set_E_bordism}) 
to say that
$(g,\underline{\alpha})\sim^{cs} (g^{\prime},\underline{\alpha}^{\prime})$
if and only if
 $$
 \underline{\alpha}^{\prime}-\underline{\alpha}=
 \underline{\mathbf{\widetilde{e}}}(\mathbb{M},g,g^{\prime})\in
\Omega^{m-1}(M,\partial M;\Theta_{M}^{\C})
 \slash\mathbf{d}\Omega^{m-2}(M,\partial M;\Theta_{M}^{\C}),$$
where
$\underline{\mathbf{\widetilde{e}}}(\mathbb{M},g,g^{\prime})$ is the secondary form
defined in (\ref{definition_Chern_Simons_relative_form_on_bordism}).  The relation $\sim^{cs}$
is an equivalence relation on $\mathbf{E}^{*}(\mathbb{M};\C)$.
\end{lem}

\begin{proof}
By Chern--Gauss--Bonnet formula, see first equality of Lemma \ref{Theorem_Chern_Gauss_Bonnet_Bordism} below, the relative Euler form
$\underline{\mathbf{e}}(\mathbb{M},g)$
from Definition \ref{definition_relative_Euler_form}
satisfies
$$
 \langle[\underline{\mathbf{e}}(\mathbb{M},g)],[1]\rangle=0.
$$
Since  $\langle\cdot,\cdot\rangle$
is non-degenerate, the relative form
$\underline{\mathbf{e}}(\mathbb{M},g)$ is exact in relative cohomology.
That is, there exists $\underline{\alpha}\in\Omega^{m-1}(M,\partial M;\Theta_{M}^{\C})$ such that
$
\mathbf{d}\underline{\alpha}=\underline{\mathbf{e}}(\mathbb{M},g).
$
Hence the space $\mathbf{E}^{*}(\mathbb{M};\C)$ is not empty.
 The relation $\sim^{cs}$ satisfies the reflexivity property, since $$\underline{\mathbf{\widetilde{e}}}(\mathbb{M},g,g)=0.$$
 Symmetry and transitivity of $\sim^{cs}$ are implied by Lemma \ref{basic_properties_of_chern_simon_forms}.
\end{proof}

\begin{defi}
 \label{definition_Co-EulerStructure_without_base_point}
 Let $\mathbf{E}^{*}(\mathbb{M};\C)$ be the space defined in (\ref{definition_Set_E_bordism}).
 The set of \textbf{co-Euler structures} on
 a bordism $\mathbb{M}$ \index{co-Euler structures on compact bordisms} is defined as the quotient
\begin{equation}
\label{definition_Set_of_coeuler_structures_bordism}
\mathfrak{Eul}^{*}(\mathbb{M};\C):=\mathbf{E}^{*}(\mathbb{M};\C)/\sim^{cs};
\end{equation}
the equivalence class of $(g,\underline{\alpha})$ will be denoted
by $[g,\underline{\alpha}]$.
\end{defi}

\begin{lem}
\label{lemma_action_is_well_defined_free_transitive_coeuleur_Structures_absolute}
Let $H^{m-1}(M,\partial M;\Theta_{M}^{\C})$ be the relative cohomology groups in degree $m-1$
with coefficients in $\Theta_{M}^{\C}$. For a closed relative form \index{closed relative form} $\underline{\beta}\in\Omega^{m-1}(M,\partial M;\Theta_{M}^{\C})$,
denote by $[\underline{\beta}]$ its corresponding class in relative cohomology.  Consider  $\Upsilon^{*}$, the action of $H^{m-1}(M,\partial M;\Theta_{M}^{\C})$
on the space of co-Euler structures $\mathfrak{Eul}^{*}(\mathbb{M};\C)$ from Definition \ref{definition_Co-EulerStructure_without_base_point}, given by \index{co-Euler structures, affine action}
\begin{equation}
\label{definition_action_of_relative_cohomology_group_on_Coeuler_str_Absolute and Relative}
\begin{array}{c}
\Upsilon^{*}:H^{m-1}(M,\partial M;\Theta_{M}^{\C})\times\mathfrak{Eul}^{*}(\mathbb{M};\C)
\rightarrow \mathfrak{Eul}^{*}(\mathbb{M};\C)\\
\left([\underline{\beta}],[g,\underline{\alpha}]\right)   \mapsto
[g,\underline{\alpha}-\underline{\beta}].
\end{array}
\end{equation}
Then, $\Upsilon^{*}$ is well defined, independent of each choice of representatives, free and
transitive on $\mathfrak{Eul}^{*}(\mathbb{M};\C)$.
\end{lem}
\begin{proof}
For $[\underline{\beta}]\in H^{m-1}(M,\partial M;\Theta_{M}^{\C})$,
a class in relative cohomology represented by
the closed relative form $\underline{\beta}\in\Omega^{m-1}(M,\partial M,\Theta_{M}^{\C})$, consider its action on
the co-Euler structure
$[g,\underline{\alpha}]$, represented by the couple $(g,\underline{\alpha})$. Remark that
$(g,\underline{\alpha}-\underline{\beta})\in\mathbf{E}^{*}(\mathbb{M};\C)$, because
$\mathbf{d}(\underline{\alpha}-\underline{\beta})=\mathbf{d}\underline{\alpha}-\mathbf{d}\underline{\beta}=
\mathbf{d}\underline{\alpha}= \underline{\mathbf{e}}(\mathbb{M},g).$
Let us prove that $\Upsilon^{*}$ does not depend on the choice of representatives. The map
$\Upsilon^{*}$ is independent of the choice of representative for the co-Euler class.
Indeed, let $(g^{\prime}, \underline{\alpha}^{\prime})$ represent
the same class as $(g,\underline{\alpha})$ in the quotient space $\mathfrak{Eul}^{*}(\mathbb{M};\C)$ for which we have
$\Upsilon^{*}_{\underline{\beta}}(g^{\prime},\underline{\alpha}^{\prime})
=(g^{\prime},(\underline{\alpha}^{\prime}-\underline{\beta}))$. Since
$
(\underline{\alpha}^{\prime}-\underline{\beta})-(\underline{\alpha}-\underline{\beta})=
(\underline{\alpha}^{\prime}-\underline{\alpha})
=\underline{\mathbf{\widetilde{e}}}(\mathbb{M},g,g^{\prime})
$
modulo relative exact forms, we have $\Upsilon^{*}_{\underline{\beta}}([g,\underline{\alpha}])=
\Upsilon^{*}_{\underline{\beta}}([g^{\prime},\underline{\alpha}^{\prime}]).$
The map $\Upsilon^{*}$ is also independent of the choice
of the representative for the class in cohomology $[\underline{\beta}]$.
Indeed,  different choices for the cohomology class of $\underline{\beta}$
are obtained by adding coboundaries in $\Omega^{m-1}(M,\partial_{+}M;\Theta_{M}^{\C})$, that is
$\underline{\beta}+\mathbf{d}\underline{\beta}^{\prime}$.
But for these forms we have
$
\Upsilon^{*}_{\underline{\beta}}([g,\underline{\alpha}])=
\Upsilon^{*}_{\underline{\beta}+\mathbf{d}\underline{\beta}^{\prime}}(
[g,\underline{\alpha}]),
$ since the equivalence relation $\sim^{cs}$ is given up to
relative exact forms only, see Lemma \ref{definition_equivalence_relation_without_base_point}. So, we have proved
$\Upsilon^{*}$ is well defined and independent of every choice of representatives.

The same argument is used to see that $H^{m-1}(M,\partial M;\Theta_{M}^{\C})$ acts freely on
$\mathfrak{Eul}^{*}(\mathbb{M};\C)$. Indeed, if $\underline{\beta}$ is such that
$[g,\underline{\alpha}-\underline{\beta}]=[g,\underline{\alpha}]$,
then
$
\underline{\beta}= \underline{\mathbf{\widetilde{e}}}(\mathbb{M},g,g)
+\mathbf{d}\underline{\beta}^{\prime},
$
but, since the first term on the right hand side in the equality above vanish,
the relative form $\underline{\beta}$ is necessarily exact.

We show this action is transitive on  $\mathfrak{Eul}^{*}(\mathbb{M};\C)$: for two classes
$[g,\underline{\alpha}]$ and $[g^{\prime},\underline{\alpha}^{\prime}]$,
we can choose the relative form
$
\underline{\beta}:=(\underline{\alpha}-\underline{\alpha}^{\prime})+
\underline{\mathbf{\widetilde{e}}}(\mathbb{M},g,g^{\prime}).
$
By Lemma \ref{basic_properties_of_chern_simon_forms}, the relative form $\underline{\beta}$ is closed, since
$
\mathbf{d}\underline{\beta}=
  \underline{\mathbf{e}}(\mathbb{M},g) -
  \underline{\mathbf{e}}(\mathbb{M},g^{\prime})+
  \mathbf{d}\underline{\mathbf{\widetilde{e}}}(\mathbb{M},g,g^{\prime})=0.
$
Finally, we have
$
\Upsilon^{*}_{[\underline{\beta}]}([g,\underline{\alpha}])=
[g,\underline{\alpha}-\underline{\beta}]=
[g^{\prime},\underline{\alpha}^{\prime}].
$
\end{proof}

\subsubsection{\textbf{The flip map for co-Euler Structures}}
\label{section_flip_map}
Let us consider the spaces of co-Euler structures $\mathfrak{Eul}^{*}(\mathbb{M};\C)$ and $\mathfrak{Eul}^{*}({\mathbb{M}^{\prime}};\C)$, from Definition
\ref{definition_Co-EulerStructure_without_base_point}, corresponding to
the mutually dual bordisms $\mathbb{M}$ and ${\mathbb{M}^{\prime}}$ respectively.
In view of Lemma \ref{Lemma_duality_e_partial_on_bordisms} in the Appendix,
there is a natural map
\begin{equation}
 \label{equation_flip_map_co-Euler_structures}
 \begin{array}{lrcl}
 \nu^{*}: &\mathfrak{Eul}^{*}(\mathbb{M};\C)
 &\rightarrow&\mathfrak{Eul}^{*}({\mathbb{M}^{\prime}};\C)\\
 &[g,\underline{\alpha}]&\mapsto&\left[g,\left(-1)^{m}\underline{\alpha}\right)\right]\\
 \end{array}
\end{equation}
which is affine over the involution in relative cohomology
$$(-1)^{m}\cdot\mathsf{id}:H^{m-1}(M,\partial M;\Theta_{M}^{\C})\rightarrow H^{m-1}(M,\partial M;\Theta_{M}^{\C}).$$

\begin{rem}
If $M$ is a closed manifold, i.e., $\partial_{+}M=\emptyset=\partial_{-}M$, then clearly
$\mathfrak{Eul}^{*}(\mathbb{M};\C)=\mathfrak{Eul}^{*}(\mathbb{M}^{\prime};\C)$, are
affine over $H^{m-1}(M;\Theta_{M}^{\C})$ and coincide with $\mathfrak{Eul}^{*}(M;\C)$,
the set of co-Euler
structures on a manifold without boundary  (see \cite{Burghelea-Haller} and \cite{Burghelea-Haller06a}).
If $M$ is closed and of odd dimension, then
the involution $\nu^{*}$, being affine over $-\mathsf{id}$, possesses a
unique fixed point in $\mathfrak{Eul}^{*}(M;\C)$, which corresponds to the
canonic co-Euler structure
$$
\mathfrak{e}^{*}_{\mathsf{can}}:=[g,(\alpha_{\mathsf{can}}=0,\alpha_{\partial}=0)]
$$
where $\alpha_{\mathsf{can}}=0$, because for odd dimensional closed manifolds $\mathbf{e}(M,g)=0$ and
forms $\alpha_{\partial}=0$, see \cite[Section 2.2]{Burghelea-Haller}.
\end{rem}

\subsubsection{\textbf{Infinitesimal variation of co-Euler structures without base point}}
\label{subsection_Variation formulas for Co-Euler structures without base point}
The following result generalizes \cite[(56)]{Burghelea-Haller}.
\begin{prop}
 \label{Lemma_variation_co-Euler_structures_on_man_with_boundary_without_base_point}
Let $\mathbb{M}$ be a bordism and assume that the relative Euler characteristics
$\chi(M,\partial_{+}M)=0$. Consider $\{(g_{u},\underline{\alpha}_{u})\}_{u}$ a smooth real one-parameter family of Riemannian metrics $g_{u}$ and relative forms
$\underline{\alpha}_{u}$,
representing the \textit{same} co-Euler structure
$[g_{u},\underline{\alpha}_{u}]\in\mathfrak{Eul}^{*} (\mathbb{M};\C)$.  For each Riemannian metric $g_{u}$ consider
the forms
$$
\underline{\mathbf{e}}(\mathbb{M},g_{u})\in\Omega^{m}(M,\partial M;\Theta_{M}^{\C})$$ and
$$B(\partial_{+}M,\partial_{-}M,g_{u})\in\Omega^{m-1}(\partial M;\Theta_{M}^{\C})$$ from
Definition \ref{definitions_of_modified_e_partial_B_forms_on_bordisms} as well as the relative Chern-Simon's
form $$\underline{\widetilde{\mathbf{e}}}(\mathbb{M},g_{u},g_{w})\in \Omega^{m-1}(M,\partial M;\Theta_{M}^{\C})$$ from Definition \ref{definition_of_Chern_Simons_secondary_classes_1}. Let $E$ be a complex flat vector bundle over $M$ with flat connection $\nabla^{E}$, endowed
with a smooth family of non-degenerate symmetric bilinears forms $b_{u}$.
If
\begin{equation}
\label{Kamber-Tondeur_form}
\omega(\nabla^{E},b_{u}):=-\frac{1}{2}\Tr(b^{-1}_{u}\nabla^{E}b_{u})\in\Omega^{1}(M;\C)
\end{equation}
denotes
the \textit{Kamber--Tondeur} form
associated to $b_{u}$ and $\nabla^{E}$, see \cite{Burghelea-Haller}, and the integral $\int_{(M,\partial M)}$ is the pairing from (\ref{pairing}). Then, the formulas
\begin{equation}
\label{equation_variation_formula_with_out_base_point_1}
\begin{array}{l}
\left.\frac{\partial}{\partial u}\right.
\int_{(M,\partial M)}
2\underline{\alpha}_{u}\wedge\omega(\nabla^{E},b_{u})
\\
\hspace{2cm}=-(-1)^{m}\int_{(M,\partial M)}\underline{\mathbf{e}}(\mathbb{M},g_{u})\Tr\left(b^{-1}_{u}\dot{b}_{u}\right)\\
\hspace{2cm}\phantom{=}+2\int_{(M,\partial M)}
\left.\frac{\partial}{\partial\tau}\right|_{\tau=0}\underline{\widetilde{\mathbf{e}}}(\mathbb{M},g_{u},g_{u}+\tau\dot{g}_{u})
\wedge\omega(\nabla^{E},b_{u})
\end{array}
\end{equation}
and
\begin{equation}
\label{equation_variation_formula_with_out_base_point_2}
\left.\frac{\partial}{\partial u}\right.\int_{\partial M}B(\partial_{+}M,\partial_{-}M,g_{u})=
\int_{\partial M}\left.\frac{\partial}{\partial\tau}\right|_{\tau=0}B(\partial_{+}M,\partial_{-}M,g+\tau \dot{g}_{u}).
\end{equation}
hold.
\end{prop}
\begin{proof}
First, remark that
$
\scriptstyle\left.\frac{\partial}{\partial u}\right.\underline{\alpha}_{u}\scriptstyle=
\left.\frac{\partial}{\partial w}\right|_{u}\left(\underline{\alpha}_{w}-
\underline{\alpha}_{u}\right)
=
\left.\frac{\partial}{\partial\tau}\right|_{0}\underline{\widetilde{\mathbf{e}}}(\mathbb{M},g_{u},g_{u}+
\tau\dot{g}_{u}),
$
and also that
$
\scriptstyle\left.\frac{\partial}{\partial u}\right.B(\partial_{+}M,\partial_{-}M,g_{u})=
\left.\frac{\partial}{\partial\tau}\right|_{0}B(\partial_{+}M,\partial_{-}M,g+\tau \dot{g}_{u}).
$
From \cite{Burghelea-Haller} we have the identity
$\scriptstyle\left.\frac{\partial}{\partial u}\right.\Tr(b^{-1}_{u}\nabla^{E}b_{u})=d\Tr\left(b^{-1}_{u}\dot{b}_{u}\right).$
Therefore,  since for each $u$, the couple $[g_{u},\underline{\alpha}_{u}]$ represents the same co-Euler structure, we obtain, modulo exact relative
forms
$$
\begin{array}{l}
\scriptstyle\left.\frac{\partial}{\partial u}\right.
\int_{(M,\partial M)}
2\underline{\alpha}_{u}\wedge\omega(\nabla^{E},b_{u})\\
\scriptstyle\hspace{0.2cm}
=\int_{(M,\partial M)}\partial_{w}|_{u}(\underline{\alpha}_{w})\wedge 2\omega(\nabla^{E},b_{u})+
\int_{(M,\partial M)}\underline{\alpha}_{w}\wedge\partial_{w}|_{u}(-\Tr(b_{w}^{-1}\nabla^{E}b_{u}))\\
\scriptstyle\hspace{0.2cm}
=2\int_{(M,\partial M)}\left.\frac{\partial}{\partial\tau}\right|_{0}\underline{\widetilde{\mathbf{e}}}(\mathbb{M},g_{u},g_{u}+
\tau\dot{g}_{u})\wedge\omega(\nabla^{E},b_{u})+
\underbrace{\scriptstyle\int_{(M,\partial M)}-\underline{\alpha}_{u}\wedge d\Tr\left(b^{-1}_{u}\dot{b}_{u}\right)}_{(*)};
\end{array}
$$
with $\underline{\alpha}_{u}=(\alpha_{u},\sigma_{u})$,
$\mathbf{d}\underline{\alpha}_{u}=\underline{\mathbf{e}}(\mathbb{M},g_{u})$ and
Stokes' Theorem, the second term on the right above
becomes
$$
\begin{array}{l}
\scriptstyle(*)=-(-1)^{m}\left(\int_{M}d\alpha_{u}\Tr\left(b^{-1}_{u}\dot{b}_{u}\right)
-\left(\int_{\partial M}i^{*}\left(\alpha_{u}\Tr\left(b^{-1}_{u}\dot{b}_{u}\right)\right)-
\int_{\partial M}\left(d^{\partial}\sigma_{u}i^{*}\Tr\left(b^{-1}_{u}\dot{b}_{u}\right)\right)\right)\right)\\

\scriptstyle\phantom{(*)}=-(-1)^{m}\left(\int_{M}d\alpha_{u}\Tr\left(b^{-1}_{u}\dot{b}_{u}\right)
-\int_{\partial M}(i^{*}\alpha_{u}-d^{\partial}\sigma_{u})i^{*}\Tr\left(b^{-1}_{u}\dot{b}_{u}\right)\right)\\





\scriptstyle\phantom{(*)}=-(-1)^{m}\int_{(M,\partial M)}\mathbf{d}\underline{\alpha}\Tr\left(b^{-1}_{u}\dot{b}_{u}\right)\\



\scriptstyle\phantom{(*)}=-(-1)^{m}\int_{(M,\partial M)}\underline{\mathbf{e}}(\mathbb{M},g_{u})\Tr\left(b^{-1}_{u}\dot{b}_{u}\right).
\end{array}
$$
\end{proof}


\subsection{Co-Euler structures with base point}
\label{section_co-Euler Structures with base point}
Here, we do not assume $\chi(M,\partial_{+}M)$ to vanish. As in the case of a closed manifold,
co-Euler structures still can be defined by introducing a base point $x_{0}$ in the interior of $M$. We denote by $\dot{M}:=M\backslash\{x_{0}\}$,
consider  $(g,\mathbf{\underline{\alpha}})$ with
$\mathbf{\underline{\alpha}}\in\Omega^{m-1}(\dot{M},\partial M;\Theta_{M}^{\C})$ as in (\ref{equation_relative_forms_with_base_point}), and define
\begin{equation}
\label{definition_sets_defining_co_Euler_structures_without_base_point}
\mathbf{E}_{x_{0}}^{*}(\mathbb{M};\C):=\left\{(g,\underline{\alpha}{})\left|
\begin{array}{rcl}
\underline{\alpha}&\in&\Omega^{m-1}(\dot{M},\partial M;\Theta_{M}^{\C})\\
\mathbf{d}\underline{\alpha}&=& \underline{\mathbf{e}}(\mathbb{M},g)\\
\end{array}
\right.
\right\}.
\end{equation}
Since $M$ is assumed to be connected, in view of the first equality 
in (\ref{Lemma_top_degree_relative_cohomology_with_base_point_vanishes}),
the space in (\ref{definition_sets_defining_co_Euler_structures_without_base_point}) is non-empty. Then, as for the case without base point,
we have the relation:
$(g,\mathbf{\underline{\alpha}})\sim^{cs}(g^{\prime},\mathbf{\underline{\alpha}}^{\prime})$ in
$\mathbf{E}_{x_{0}}^{*}(\mathbb{M};\C)$ if and only if
    \begin{equation}
     \label{equation_definition_of_equ_relation_co-Euler_structures_with_base_point}
  \underline{\alpha}^{\prime}-\underline{\alpha}=
  \underline{\mathbf{\widetilde{e}}}(\mathbb{M},g,g^{\prime})\quad\in\text{ }
  \Omega^{m-1}(\dot{M},\partial M;\Theta_{M}^{\C})\slash\mathbf{d}\Omega^{m-2}(\dot{M},\partial M;\Theta_{M}^{\C}).
    \end{equation}
The relation in (\ref{equation_definition_of_equ_relation_co-Euler_structures_with_base_point})
is an equivalence relation for the same reasons as in the case without base point.

\begin{defi}
The quotient
space
\label{definition_Set_of_coeuler_structures_with_base_point}
$$
\begin{array}{l}
 \mathbf{\mathfrak{Eul}}_{x_{0}}^{*}(\mathbb{M};\C):=\mathbf{E}_{x_{0}}^{*}(\mathbb{M};\C)/\sim^{cs}
\end{array}
$$
is called the space of co-Euler structures based at $x_{0}$ on $\mathbb{M}$ and
the equivalence class of the pair $(g,\underline{\alpha})$ is denoted by 
$[g,\underline{\alpha}]$.
\end{defi}

The action of $H^{m-1}(M,\partial M;\Theta_{M}^{\C})$ on
$\mathbf{\mathfrak{Eul}}_{x_{0}}^{*}(\mathbb{M};\C)$ defined by
\begin{equation}
\label{definition_action_of_relative_cohomology_group_on_Coeuler_str_Absolute and Relative_with_base_point}
\begin{array}{c}
\Upsilon^{*}:H^{m-1}(M,\partial M;\Theta_{M}^{\C})\times\mathbf{\mathfrak{Eul}}_{x_{0}}^{*}(\mathbb{M};\C)
\rightarrow \mathbf{\mathfrak{Eul}}_{x_{0}}^{*}(\mathbb{M};\C)\\
\left([\underline{\beta}],[g,\underline{\alpha}]\right)   \mapsto
[g,\underline{\alpha}-\underline{\beta}]
\end{array}
\end{equation}
is well defined and independent of each choice of representatives,
see Lemma \ref{lemma_action_is_well_defined_free_transitive_coeuleur_Structures_absolute}.
In addition, the action specified by (\ref{definition_action_of_relative_cohomology_group_on_Coeuler_str_Absolute and Relative_with_base_point})
is free and transitive since
$H^{m-1}(M,\partial M)\cong H^{m-1}(\dot{M},\partial M)$, see
(\ref{Lemma_top_degree_relative_cohomology_with_base_point_vanishes}).

Finally, the
flip map
\begin{equation}
 \label{equatio}
 \begin{array}{lrcl}
 \nu^{*}: &\mathbf{\mathfrak{Eul}}_{x_{0}}^{*}(\mathbb{M};\C)
 &\rightarrow&\mathbf{\mathfrak{Eul}}_{x_{0}}^{*}({\mathbb{M}^{\prime}};\C)\\
 &[g,\underline{\alpha}]&\mapsto&\left[g,\left(-1)^{m}\underline{\alpha}\right)\right]\\
 \end{array}
\end{equation}
intertwines the spaces $\mathbf{\mathfrak{Eul}}_{x_{0}}^{*}(M,\partial_{\pm}M,\partial_{\mp}M;\C)$ and
it is affine over the involution in relative cohomology
$$(-1)^{m}\mathsf{id}:H^{m-1}(\dot{M},\partial M;\Theta_{M}^{\C})\rightarrow H^{m-1}(\dot{M},\partial M;\Theta_{M}^{\C}).$$

\subsubsection{\textbf{Variational formula for co-Euler structures with base point}}
\label{subsection_Variation formulas for Co-Euler structures with base point}
We give an analog to Proposition
\ref{Lemma_variation_co-Euler_structures_on_man_with_boundary_without_base_point} in the case of co-Euler structures with base point.
Let $\alpha\in\Omega^{m-1}(\dot{M};\Theta_{M}^{\C})$ be a smooth differential form on $M$,
with possible singularity $x_{0}\in\mathbf{int}(M)$,
the interior of $M$ and $\dot{M}:=M\backslash\{x_{0}\}.$
For
$\omega$ a closed $1$-form on $M$, we make sense of integrals of the type
$\int_{M}\alpha\wedge\omega,$
by means of a \textit{regularization procedure} \index{regularization procedure} as described in the remaining of this section.

First, recall that the \textit{local degree}  of $\alpha$ at the singularity $x_{0}$, see
for instance \cite[Chapter II.11]{Bott-Tu}, is given by
\begin{equation}
\label{Local_degree}
\mathsf{deg}_{x_{0}}(\alpha):=\lim_{\delta\rightarrow 0}\int_{\partial(\mathbb{B}^{m}(\delta,x))}i^{*}\alpha,
\end{equation}
where $\partial(\mathbb{B}^{m}(\delta,x))$ indicates the boundary
of the $m$-dimensional closed ball $\mathbb{B}^{m}(\delta,x)$ centered at $x_{0}$ and radius $\delta>0$. With the
standard sign convention involved in Stokes' Theorem, $\partial(\mathbb{B}^{m}(\delta,x))$ is oriented
with respect to the unit outwards pointing vector field normal
to $\mathbb{B}^{m}(\delta,x)$.

\begin{lem}
\label{Lemma_main_properties_of_S_in_particular_linearity_independance_and_how_it_acts_on_exact_forms}
Let $\alpha$ be a smooth form in $\Omega^{m-1}(\dot{M};\Theta_{M}^{\C})$
such that
$d\alpha$ and $\alpha_{\partial}$ are smooth and without singularities in $M$. For $\omega$ a smooth closed $1$-form on $M$, choose
a smooth function $f\in C^{\infty}(M)$ such that the $1$-form
$$
\omega^{\prime}:=\omega-df
$$
is smooth on $M$ and vanishes on a small neighborhood of $x_{0}$.
Then the complex-valued function 
\begin{equation}
 \label{definition_function_S}
\mathcal{S}(\underline{\alpha},\omega,f):=\int_{(M,\partial M)}\underline{\alpha}\wedge\omega^{\prime}+
(-1)^{m}\int_{(M,\partial M)}\mathbf{d}\underline{\alpha}\wedge f
- f(x_{0})\mathsf{deg}_{x_{0}}(\alpha),
\end{equation}
does not depend on the choice of $f$ and satisfies the following assertions.
\begin{enumerate}
\item
If $\underline{\beta}\in\Omega^{m-1}(M,\partial M;\Theta_{M}^{\C})$, i.e., without singularities, then
$$
\mathcal{S}(\underline{\beta},\omega)=\int_{(M,\partial M)}\underline{\beta}\wedge\omega.
$$
In particular, $\mathcal{S}(\mathbf{d}\underline{\gamma},\omega)=0$
for all $\underline{\gamma}\in\Omega^{m-2}(M,\partial M;\Theta_{M}^{\C})$.
\item
$\mathcal{S}(\omega,\underline{\alpha})$ is linear in $\underline{\alpha}$ and in $\omega$.
\item
$\mathcal{S}(\underline{\alpha},dh)=(-1)^{m}\int_{(M,\partial M)}\mathbf{d}\underline{\alpha}\wedge h
-h(x)\mathsf{deg}_{x_{0}}(\alpha)$.
\end{enumerate}
\end{lem}
\begin{proof}
Without loss of generality assume $\mathcal{X}(\alpha)=\{x\}$.
We want to know how the function $\int_{(M,\partial M)}\underline{\alpha}\wedge\omega^{\prime}$ changes,
with respect to $f$. Let us take $f_{1},f_{2}\in C^{\infty}(M)$ two functions as above, such that
the corresponding one forms $\omega^{\prime}_{1},\omega^{\prime}_{2}$
vanish on a small open neighborhood of $x_{0}$, so that $d(f_{2}-f_{1})=0$ locally around $x_{0}$; that means
$f_{2}-f_{1}$ is  constant\footnote{If we choose $f_{2}(x)=f_{1}(x)=0$, then $f_{2}-f_{1}=0$ around $x_{0}$.}
on a small neighborhood of $x_{0}$. Now, consider the variation
$$
\begin{array}{lll}
\scriptstyle\Delta&\scriptstyle=&\scriptstyle\int_{(M,\partial M)}\underline{\alpha}\wedge(w_{2}^{\prime}-w_{1}^{\prime})\\
&\scriptstyle=&\scriptstyle\int_{M\backslash\{x\}}\alpha\wedge(w^{\prime}_{2}-w^{\prime}_{1})-\int_{\partial M}\alpha_{\partial}\wedge i^{*}(w^{\prime}_{2}-w^{\prime}_{1})\\
&\scriptstyle=&\scriptstyle-\int_{M\backslash\{x\}}\alpha\wedge d(f_{2}-f_{1})+\int_{\partial M}\alpha_{\partial}\wedge i^{*}(d(f_{2}-f_{1})),
\end{array}
$$
We develop both terms on the right of the last equality above.
The first one, the integral over $M$, can be re written as
$$
\scriptstyle-\int_{M\backslash\{x\}}
\alpha\wedge d(f_{2}-f_{1})=-(-1)^{m-1}\int_{M\backslash\{x\}}d(\alpha(f_{2}-f_{1}))
+(-1)^{m-1}\int_{M\backslash\{x\}}d\alpha\wedge(f_{2}-f_{1}),
$$
whereas the second one, the integral over the boundary becomes
$$
\begin{array}{lll}
\scriptstyle\int_{\partial M}\alpha_{\partial}\wedge d^{\partial} i^{*}(f_{2}-f_{1})&\scriptstyle=&
\scriptstyle(-1)^{m-2}\underbrace{\scriptstyle\int_{\partial M}d^{\partial}(\alpha_{\partial}\wedge i^{*}(f_{2}-f_{1}))}_{=0}
-(-1)^{m-2}\int_{\partial M}d^{\partial}\alpha_{\partial}\wedge i^{*}(f_{2}-f_{1})\\
&\scriptstyle=&\scriptstyle(-1)^{m-1}\int_{\partial M}d^{\partial}\alpha_{\partial}\wedge i^{*}(f_{2}-f_{1})
\end{array}
$$
and therefore
$$
\begin{array}{lll}
\scriptstyle\Delta
&\scriptstyle=&\scriptstyle-(-1)^{m-1}\left(\int_{M\backslash\{x\}}d(\alpha(f_{2}-f_{1}))-\int_{M\backslash\{x\}}d\alpha\wedge(f_{2}-f_{1})
-\int_{\partial M}d^{\partial}\alpha_{\partial}\wedge i^{*}(f_{2}-f_{1})\right)\\
&\scriptstyle=&\scriptstyle-(-1)^{m-1}\left(\int_{M\backslash\{x\}}d(\alpha(f_{2}-f_{1}))-\int_{M}d\alpha\wedge(f_{2}-f_{1})
-\int_{\partial M}d^{\partial}\alpha_{\partial}\wedge i^{*}(f_{2}-f_{1})\right),
\end{array}
$$
where we have used $$\scriptstyle\int_{M\backslash\{x\}}d\alpha\wedge(f_{2}-f_{1})=\int_{M}d\alpha\wedge(f_{2}-f_{1}),$$
since by assumption, the form $d\alpha$ does not have singularities on $M$. Hence,
to make sense of $\Delta$, we now make sense of the integral
$
\int_{M\backslash\{x\}}d(\alpha(f_{2}-f_{1})).
$
This integral can be computed as the limit:
$$
\scriptstyle\int_{M\backslash\{x\}}d(\alpha(f_{2}-f_{1})):=
\lim_{\delta\rightarrow 0}\int_{M\backslash\mathbb{B}(\delta,x)}d(\alpha(f_{2}-f_{1}))
$$
where $\mathbb{B}(\delta,x)$ is the closed ball centered at $x_{0}$ of radius $\delta>0$ and with boundary
$\partial(\mathbb{B}(\delta,x))$ endowed with the orientation
specified by the unit outwards pointing vector field normal to $\mathbb{B}(\delta,x)$. Then, by using
Stokes' Theorem with the standard convention, the limit
above can be computed as
$$
\begin{array}{lll}
\scriptstyle\int_{M\backslash\{x\}}d(\alpha(f_{2}-f_{1}))&\scriptstyle=&\scriptstyle
\lim_{\delta\rightarrow 0}\int_{M\backslash\mathbb{B}(\delta,x)}d(\alpha(f_{2}-f_{1}))\\
&\scriptstyle=&\scriptstyle\lim_{\delta\rightarrow 0}
\int_{\partial(M\backslash\mathbb{B}(\delta,x))}i^{*}(\alpha(f_{2}-f_{1}))\\
&\scriptstyle=&\scriptstyle
\lim_{\delta\rightarrow 0}
\left(\int_{\partial M}i^{*}(\alpha(f_{2}-f_{1}))+
\int_{-\partial(\mathbb{B}^{m}(\delta,x))}i^{*}(\alpha(f_{2}-f_{1}))\right)\\
&\scriptstyle=&\scriptstyle
\int_{\partial M}i^{*}(\alpha(f_{2}-f_{1}))+
\lim_{\delta\rightarrow 0}\int_{-\partial(\mathbb{B}^{m}(\delta,x))}i^{*}(\alpha(f_{2}-f_{1})),\\
\end{array}
$$
where $-\partial(\mathbb{B}^{m}(\delta,x))$ indicates
the sphere with opposite orientation as that of $\partial(\mathbb{B}(\delta,x))$. 
Now, we look at the second term on the right of the equality above. Since $f_{2}-f_{1}$ is constant
on a small neighborhood of $x_{0}$, we have, for $\delta^{\prime}>0$ small enough,
$$
\begin{array}{rl}
\scriptstyle\lim_{\delta\rightarrow 0}\int_{-\partial(\mathbb{B}^{m}(\delta,x))}i^{*}(\alpha(f_{2}-f_{1}))&\scriptstyle=
\scriptstyle(f_{2}-f_{1})(x^{\prime})\lim_{\delta\rightarrow 0}\int_{-\partial(\mathbb{B}^{m}(\delta,x))}i^{*}\alpha\quad\scriptstyle\text{for all }
x^{\prime}\in\mathbb{B}(\delta^{\prime},x),\\
&\scriptstyle=\scriptstyle
(-1)^{m}(f_{2}-f_{1})(x)\mathsf{deg}_{x_{0}}(\alpha),
\end{array}
$$
where the sign $(-1)^{m}$ above comes from the standard convention taken for the Stokes' Theorem.
Hence
$$
\scriptstyle\int_{M\backslash\{x\}}d(\alpha(f_{2}-f_{1}))=\int_{\partial M}i^{*}(\alpha(f_{2}-f_{1}))+
(-1)^{m}(f_{2}-f_{1})(x)\mathsf{deg}_{x_{0}}(\alpha).
$$
Therefore the variation $\Delta$ becomes
$$
\begin{array}{lll}
\scriptstyle\Delta&\scriptstyle=&\scriptstyle-(-1)^{m-1}\left[\int_{\partial M}i^{*}(\alpha(f_{2}-f_{1}))+(-1)^{m}(f_{2}-f_{1})(x)\mathsf{deg}_{x_{0}}(\alpha)\right.\\
\scriptstyle&\scriptstyle&\scriptstyle \hspace{5cm}\left.-\int_{M}d\alpha\wedge(f_{2}-f_{1})
-\int_{\partial M}d^{\partial}\alpha_{\partial}\wedge i^{*}(f_{2}-f_{1})\right]\\
&\scriptstyle=&\scriptstyle-(-1)^{m-1}\left[\int_{\partial M}(i^{*}\alpha-d^{\partial}\alpha_{\partial})\wedge i^{*}(f_{2}-f_{1})
-\int_{M}d\alpha\wedge(f_{2}-f_{1})\right.+
\scriptstyle\left.(-1)^{m}(f_{2}-f_{1})(x)\mathsf{deg}_{x_{0}}(\alpha)\right]\\
&\scriptstyle=&\scriptstyle-(-1)^{m-1}\left[-\int_{(M,\partial M)}\mathbf{d}\underline{\alpha}\wedge(f_{2}-f_{1})
+(-1)^{m}(f_{2}-f_{1})(x)\mathsf{deg}_{x_{0}}(\alpha)\right]\\
&\scriptstyle=&\scriptstyle-\left((-1)^{m}\int_{(M,\partial M)}\mathbf{d}\underline{\alpha}\wedge(f_{2}-f_{1})
-(f_{2}-f_{1})(x)\mathsf{deg}_{x_{0}}(\alpha)\right),
\end{array}
$$
and so
$$
\begin{array}{rl}
\scriptstyle\mathcal{S}_{f_{2}}(\underline{\alpha},\omega)-\mathcal{S}_{f_{1}}(\underline{\alpha},\omega)
&\scriptstyle=\Delta+\left((-1)^{m}\int_{(M,\partial M)}\mathbf{d}\underline{\alpha}\wedge(f_{2}-f_{1})-
(f_{2}-f_{1})(x)\mathsf{deg}_{x_{0}}(\alpha)\right)\\
&\scriptstyle=0,\\
\end{array}
$$
so
$
\mathcal{S}_{f}(\underline{\alpha},\omega)
$
does not depend on the choice of $f$.
Remark the
linearity of $\mathcal{S}(\underline{\alpha},\omega)$ with respect to $\omega$ immediately
follows also from its independance of $f$. The remaining assertions in
(1) and (2) follow from similar considerations as above, we omit the details.
Let us turn to assertion (3).
For a smooth function $h$, we compute
$$
\begin{array}{rcl}
\scriptstyle\mathcal{S}_{f}(\underline{\alpha},\omega+dh)&\scriptstyle=&\scriptstyle\int_{(M,\partial M)}\underline{\alpha}\wedge(\omega+dh-df)+
(-1)^{m}\int_{(M,\partial M)}\mathbf{d}\underline{\alpha}\wedge f-f(x_{0})\mathsf{deg}_{x_{0}}(\alpha),\\
&\scriptstyle=&\scriptstyle\int_{(M,\partial M)} \underline{\alpha}\wedge(\omega+d(h-f))\\
&\scriptstyle&\scriptstyle\hspace{1cm}+(-1)^{m}\int_{(M,\partial M)}\left(\mathbf{d\underline{\alpha}}\wedge(f-h)+\mathbf{d\underline{\alpha}}\wedge h\right)\\
&&\scriptstyle\hspace{4cm}
-\left(f(x_{0})-h(x)\right)\mathsf{deg}_{x_{0}}(\alpha)-h(x)\mathsf{deg}_{x_{0}}(\alpha),\\
&\scriptstyle=&\scriptstyle\int_{(M,\partial M)}\underline{\alpha}\wedge(\omega-d(f-h))\\
&\scriptstyle&\scriptstyle\hspace{1cm}+(-1)^{m}\int_{(M,\partial M)}\mathbf{d\underline{\alpha}}\wedge(f-h)
-(f-h)(x)
\mathsf{deg}_{x_{0}}(\alpha)+\\
&&\scriptstyle\hspace{4.5cm}
(-1)^{m}\int_{(M,\partial M)}\mathbf{d\underline{\alpha}}\wedge h-h(x)\mathsf{deg}_{x_{0}}(\alpha),\\
&\scriptstyle=&\scriptstyle\mathcal{S}_{f-h}(\underline{\alpha},\omega)+(-1)^{m}\int_{(M,\partial M)}\mathbf{d\underline{\alpha}}\wedge h-h(x)
\mathsf{deg}_{x_{0}}(\alpha).
\end{array}
$$
that is,
$$
\begin{array}{lll}
 \scriptstyle(-1)^{m}\int_{(M,\partial M)}\mathbf{d\underline{\alpha}}\wedge h-h(x)\mathsf{deg}_{x_{0}}(\alpha)&\scriptstyle=&
\scriptstyle\mathcal{S}_{f}(\underline{\alpha},\omega+dh)-\mathcal{S}_{f-h}(\underline{\alpha},\omega)\\
&\scriptstyle=&\scriptstyle\mathcal{S}_{f}(\underline{\alpha},\omega+dh)-\mathcal{S}_{f}(\underline{\alpha},\omega)\\
&\scriptstyle=&\scriptstyle\mathcal{S}_{f}(\underline{\alpha},dh)\\
&\scriptstyle=&\scriptstyle\mathcal{S}(\underline{\alpha},dh),
\end{array}
$$
where the second equality above holds, since $\mathcal{S}$ is independent of $f$ and the third one because
$\mathcal{S}$ is linear on $\omega$.
\end{proof}
\begin{cor}
 \label{corollary_alpha_computes_local_degre}
 Let $\alpha$ be as in Lemma \ref{Lemma_main_properties_of_S_in_particular_linearity_independance_and_how_it_acts_on_exact_forms}.
 Then,
 we have the formula
 $$\mathsf{deg}_{x_{0}}(\alpha)=(-1)^{m}\int_{(M,\partial M)}\mathbf{d}\underline{\alpha}.$$
\end{cor}
\begin{proof}
Let $\omega$, $f$, $\underline{\alpha}$ and $\mathcal{X}(\alpha)$ be
as above and consider $f_{0}$ to be a constant function on $M$. Then we compute
$$
\begin{array}{rl}
\scriptstyle\mathcal{S}_{f+f_{0}}(\underline{\alpha},\omega)&\scriptstyle=
\int_{(M,\partial M)}\underline{\alpha}\wedge\omega^{\prime}\\
&\scriptstyle\hspace{1cm}+(-1)^{m}\int_{(M,\partial M)}\mathbf{d}\underline{\alpha}\wedge f\\
&\scriptstyle\hspace{3cm}+(-1)^{m}f_{0}\int_{(M,\partial M)}\mathbf{d}\underline{\alpha}-
\scriptstyle\left( f(x_{0})\mathsf{deg}_{x_{0}}(\alpha)
+f_{0} \mathsf{deg}_{x_{0}}(\alpha)\right),\\ \\
&\scriptstyle=\mathcal{S}_{f}(\underline{\alpha},\omega)+f_{0}\left((-1)^{m}\int_{(M,\partial M)}\mathbf{d}\underline{\alpha}-
 \mathsf{deg}_{x_{0}}(\alpha)\right).
\end{array}
$$
But, from Lemma \ref{Lemma_main_properties_of_S_in_particular_linearity_independance_and_how_it_acts_on_exact_forms} above,
we know $\mathcal{S}_{f+f_{0}}(\underline{\alpha},\omega)=\mathcal{S}_{f}(\underline{\alpha},\omega)$,
and hence the last term on the right above vanishes, so that
the desired relation between the total degree of the form $\alpha$ and $\underline{\alpha}$ follows.
\end{proof}
The formula obtained in Corollary \ref{corollary_alpha_computes_local_degre}, which
computes the total degree of $\alpha$ in terms of the relative form $\underline{\alpha}$, is used
to conclude the following, and hence generalizing formula
(\ref{equation_variation_formula_with_out_base_point_1})
in Proposition \ref{Lemma_variation_co-Euler_structures_on_man_with_boundary_without_base_point}.

\begin{prop}
\label{Proposition_regularization_function_S_both_boundaries}
Consider a
bordism $\mathbb{M}$, together with
the relative Euler form $\underline{\mathbf{e}}(\mathbb{M},g)$ from
Definition \ref{definition_relative_Euler_form}.
For $x_{0}$ a base point in the interior of $M$, consider
the space $\mathbf{\mathfrak{Eul}}_{x_{0}}^{*}(\mathbb{M};\C)$ of co-Euler structures at $x_{0}$ from Definition \ref{definition_Set_of_coeuler_structures_with_base_point}.
Let $\mathfrak{e}^{*}\in \mathbf{\mathfrak{Eul}}_{x_{0}}^{*}(\mathbb{M};\C)$ be represented by
$(g,\underline{\alpha})$, where $\underline{\alpha}:=(\alpha,\alpha_{\partial})$ is a relative form with
$\alpha\in\Omega^{m-1}(\dot{M};\Theta_{M}^{\C})$
with unique singularity at $x_{0}$ and assume that
$d\alpha$ and $\alpha_{\partial}$ are smooth, i.e. without singularities in $M$.
For $\omega\in\Omega^{1}(M)$, a smooth closed $1$-form on $M$, choose
a smooth function $f\in C^{\infty}(M)$ such that
$
\omega^{\prime}:=\omega-df\in\Omega^{1}(M)
$ is a smooth $1$-form that vanishes on a small neighborhood of $x_{0}$. Then
\begin{equation}
\label{equation_function_S_both_boundaries}
\begin{array}{rl}
\mathcal{S}_{f}(\underline{\alpha},\omega)=&\int_{(M,\partial M)}\underline{\alpha}\wedge\omega^{\prime}+
(-1)^{m}\int_{(M,\partial M)}\underline{\mathbf{e}}(\mathbb{M},g)\wedge f\\
&-f(x_{0})\chi(M,\partial_{-}M)
\end{array}
\end{equation}
In particular, if $\mathfrak{e}^{*}$ is represented by $(g,\underline{\alpha})$ and
 $(g^{\prime},\underline{\alpha}^{\prime})$, then
\begin{equation}
\label{Lemma_variation_of_S_b_constant}
\mathcal{S}(\underline{\alpha}^{\prime},\omega)-\mathcal{S}(\underline{\alpha},\omega)=
 \int_{(M,\partial M)}\underline{\mathbf{\widetilde{e}}}(\mathbb{M},g,g^{\prime})\wedge\omega.
\end{equation}
\end{prop}

\begin{proof} Under these assumption, from Corollary \ref{corollary_alpha_computes_local_degre}, we have
    $$
    \scriptstyle\mathsf{deg}_{x_{0}}(\alpha)=(-1)^{m}\int_{(M,\partial M)}\mathbf{d}\underline{\alpha}=
    (-1)^{m} \int_{(M,\partial M)}\underline{\mathbf{e}}(\mathbb{M},g)= \chi(M,\partial_{-}M),
    $$
    where the last equality follows from Gauss--Bonnet--Chern Theorem. Therefore,
    (\ref{equation_function_S_both_boundaries}) follows
    from the definition of $\mathcal{S}$ in (\ref{definition_function_S}).
    Finally, formula (\ref{Lemma_variation_of_S_b_constant}) follows from (\ref{equation_function_S_both_boundaries})
    and
    the defining relation (\ref{equation_definition_of_equ_relation_co-Euler_structures_with_base_point}).

\end{proof}

 \section{Poincar\'e duality for (co)-Euler Structures}
 \label{Section_Euler_structure_and_Co-Euler_Structures}
 \subsection{Euler Structures on bordisms}
 Let $\mathbb{M}$ be a compact Riemannian bordism of dimension $m$. Euler structures were introduced by Turaev in \cite{Turaev90} in order to remove the metric ambiguities in the definition of the Reidemeister torsion. In this section, we recall a possible definition adapted to our conventions. For the sake of brevity, we assume that $\chi(M,\partial_{-}M)=0$ and we restrict to the case of Euler structure without base point. The general case, without any assumption on $\chi(M,\partial_{-}M)$, leads to the definition of Euler structures with a base point $x_{0}$ in the interior of $M$ in an analog manner as in the situation for closed manifolds, see \cite{Burghelea-Haller06a} and \cite{Burghelea-Haller06b}.
 \begin{defi}
 \label{definition_adapted_vector_field_bordism}
 Let $X:M\rightarrow TM$ be a vector field on $M$, which is transverse to the zero section, inward pointing along $\partial_{+}M$ and outward pointing along $\partial_{-}M$. We call such a vector field $X$ to be \emph{adapted} to the bordism $\mathbb{M}$.
 \end{defi}
 Let $\mathcal{X}=X^{-1}(0)$ be the set of zeros of $X$. The transversality condition means each $x\in\mathcal{X}$ is non-degenerate with Hopf index $\mathsf{Ind}_{X}(x)\in\{\pm 1\}$. Consider the singular $0$-chain in $M$
 \begin{equation}
  \label{relative_Euler_chain}
   \mathbf{e}(X):=\sum_{x\in\mathcal{X}}\mathsf{Ind}_{X}(x) x\in C_{0}^{\text{sing}}(M;\C)
 \end{equation}
 If $\chi(M,\partial_{-}M)=0$, then by the Hopf--Poincar\'e theorem, see for instance \cite[Lemma 1.2 and Proposition 1.1]{Benedetti-Petronio}, $X$ admits a (smooth) singular $1$-chain $c\in C_{1}^{\text{sing}}(M;\C)$ such that
 \begin{equation}
  \label{relative_one_chain}
 \partial c= \mathbf{e}(X);
 \end{equation}
 this singular $1$-chain will be called an \textit{Euler chain}, see \cite[Section 4]{Burghelea-Haller06b}.

 Let $X_{0}$ and $X_{1}$ be two adapted vector fields to $\mathbb{M}$.
 Then, there exists a smooth one-parameter family of vector fields $X_{t}$ connecting $X_{0}$ to $X_{1}$, with the property that $X_{t}$ is inward pointing along $\partial_{+}M$ and outward pointing along $\partial_{-}M$, for each $t\in I:=[0,1]$. For $p_{M}:M\times I\rightarrow M$, the canonical projection, consider the bundle $p_{M}^{*}TM\rightarrow M\times I$ and denote by $\widetilde{X}\in\Gamma(p_{M}^{*}TM)$ the section
 corresponding to the smooth family of vector fields $X_{t}$. With the help of small perturbations, we may assume that $\widetilde{X}$ is also transversal to the zero section; in other words that $X_{t}$ is adapted to $\mathbb{M}$ for each $t\in I$. Therefore, its zero set $\widetilde{\mathcal{X}}:=\widetilde{X}^{-1}(0)\subset M\backslash \partial M\times I$ is a canonically oriented one dimensional submanifold with boundary
 $\partial\widetilde{\mathcal{X}}=\widetilde{\mathcal{X}}\cap(M\times\partial I)$. Let
 \begin{equation}
  \label{defining_equivalence_class_one_chains}
 \widetilde{\mathbf{e}}(X_{0},X_{1})\in C_{1}^{\text{sing}}(M;\C)/\partial C_{2}^{\text{sing}}(M;\C)
 \end{equation}
 be the equivalence class, called the \textit{Chern--Simons' class}, obtained by projecting
 a representative of the fundamental class of $\widetilde{\mathcal{X}}$ into $M$, by means of the profection $p_{M}$.
 The class $\widetilde{\mathbf{e}}(X_{0},X_{1})$ depend neither on the representative of the fundamental class of $\widetilde{\mathcal{X}}$ nor on the homotopy of vector fields connecting $X_{0}$ to $X_{1}$,
 see \cite{Burghelea-Haller06a}. The class $\widetilde{\mathbf{e}}(X_{0},X_{1})$ is represented by the $0$-set of a generic homotopy connecting $X_{0}$ to $X_{1}$, by a smooth family of vector fields $X_{t}$ adapted to
 $\mathbb{M}$ for each $t\in [0,1]$, so that  the integral
 \begin{equation}
 \label{integration_over_chern-simons_cochain}
 \int_{\widetilde{\mathbf{e}}(X_{0},X_{1})}\omega=\int_{\widetilde{\mathcal{X}}}p_{M}^{*}\omega
 \end{equation}
 is well defined, for every closed one form $\omega\in\Omega^{1}(M)$. Moreover, for the Chern--Simons' classes in (\ref{defining_equivalence_class_one_chains}), and the singular 0-chains in
 (\ref{relative_Euler_chain}), the relations
 \begin{equation}
  \label{defining_relation_Chern_Simons_vectorfield}
  \begin{array}{l}
  \widetilde{\mathbf{e}}(X,X)=0,\\
   \widetilde{\mathbf{e}}(X_{0},X_{1})+\widetilde{\mathbf{e}}(X_{1},X_{2})=\widetilde{\mathbf{e}}(X_{0},X_{2})\\
   \partial\widetilde{\mathbf{e}}(X_{0},X_{1})=\mathbf{e}(X_{1})-\mathbf{e}(X_{0}),\\
   \end{array}
 \end{equation}
 hold.

 Now, for simplicity assume $\chi(M,\partial_{-}M)=0$. Let
 $(X_{0},c_{0})$ and $(X_{1},c_{1})$ be two pairs of  adapted vector fields $X_{0}$ and $X_{1}$ with corresponding singular $1$-chain $c_{0}$ and $c_{1}$ as in (\ref{relative_one_chain}), respectively.
 We call such pairs to be equivalent if and only if
 \begin{equation}
  \label{defining_relation_one_chain}
 c_{1}-c_{0}=\widetilde{\mathbf{e}}(X_{0},X_{1})\in C_{1}^{\text{sing}}(M;\C)/\partial C_{2}^{\text{sing}}(M;\C).
 \end{equation}
 This is an equivalence relation because of the identities in (\ref{defining_relation_Chern_Simons_vectorfield}) and we denote by $[X,c]$ the corresponding equivalence classes.

 \begin{defi}
 \label{def_space_of_Euler_structures}
  Assume $\chi(M;\partial_{-}M)=0$. The space of \textit{Euler structures} $\mathbf{\mathfrak{Eul}}(\mathbb{M};\C)$ is defined as the set of equivalence classes $[X,c]$ of pairs $(X,c)$ under the equivalence relation in
  (\ref{defining_relation_one_chain}).
 \end{defi}
 There is an action $\Upsilon$ of $H_{1}(M;\C)$ on $\mathbf{\mathfrak{Eul}}(\mathbb{M};\C)$, given by $$\Upsilon: ([X,c],[\sigma])\mapsto [X,c+\sigma],$$ for each $\sigma \in H_{1}(M;\C)$ on $[X,c]\in \mathbf{\mathfrak{Eul}}(\mathbb{M};\C)$. This action is well defined, free and transitive, because of the relations in (\ref{defining_relation_Chern_Simons_vectorfield}), see also \cite{Burghelea-Haller06a}, \cite{Benedetti-Petronio} and \cite{Turaev90}.

  Recall that, see \cite[Chapter II.11]{Bott-Tu}, under the involution
  \begin{equation}
  \label{involution_at_boundary}
  \xi:TM\rightarrow TM\quad\text{given by}\quad \xi(Y)=-Y
  \end{equation}
  the Hopf index of $x\in\mathcal{X}$ satisfies $\mathsf{Ind}_{-X}(x)=(-1)^{m}\mathsf{Ind}_{X}(x)$, and hence
  \begin{equation}
  \label{involution_effect_on_e_and_e_tilde}
  \mathbf{e}(- X)=(-1)^{m}\mathbf{e}(X)\quad\text{and}\quad\widetilde{\mathbf{e}}(- X_{1},- X_{2})=(-1)^{m}\widetilde{\mathbf{e}}(X_{1}, X_{2}),
  \end{equation}
  so that, we obtain a flip map, between Euler structures on dual bordisms
 \begin{equation}
  \label{action_Euler_structures_affine_homology}
 \begin{array}{rrcl}
 \nu:&\mathbf{\mathfrak{Eul}}(\mathbb{M};\C)&\rightarrow&\mathbf{\mathfrak{Eul}}({\mathbb{M}^{\prime}};\C)\\
     &[X,c]&\mapsto& [-X,(-1)^{m}c],
 \end{array}
 \end{equation}
 which is affine over the involution in homology
 \begin{equation}
  \label{Flip_map_Euler_structures}
 (-1)^{m}\mathsf{id}:H_{1}(M,\C)\rightarrow H_{1}(M,\C).
 \end{equation}

 \subsection{A relative Mathai--Quillen form}
 \label{section_A relative Mathai Quillen form}
 Let $\mathbb{M}$ be a bordism of dimension $m$ and Riemannian metric $g$. For $\pi:TM\rightarrow M$, recall that the Mathai--Quillen form
  \begin{equation}
  \label{Mathai_Quillen_form}
  \psi(M,g)\in\Omega^{m-1}(TM\backslash M;\pi^{*}\Theta_{M})
  \end{equation}
 associated to the Levi--Civit\`a connection on $TM$,
 satisfies
  \begin{equation}
  \label{derivative_Quillen_form_equals_pull_bakc_Euler_form}
 d\psi(M,g)=\pi^{*}\mathbf{e}(M,g)
  \end{equation}
 where $\mathbf{e}(M,g)$ is the Euler form of $M$, and for (\ref{involution_at_boundary})
 \begin{equation}
  \label{Change_Mathai_Quillen_form_wrt_involution}
 \psi(M,g)=(-1)^{m}\xi^{*}\psi(M,g),
  \end{equation}
 see \cite{Mathai-Quillen} and \cite{Bismut-Zhang}.

 \begin{defi}
 \label{Relative_Mathai_Quillen_form}
 Let $Q\subseteq TM|_{\partial M}$ be the subset of all vectors over
 $\partial_{+}M$ which are inward pointing and all vectors over $\partial_{-}M$ which are outward pointing.
 We define a  \textit{relative Mathai--Quillen form} by
$$ \underline{\psi}(\mathbb{M},g):=(\psi(M,g),\psi_{\partial}(\partial_{+}M,\partial_{-}M,g))\in \Omega^{m-1}(TM\backslash M,Q;\pi^{*}\Theta_{M})$$ where
its boundary component
 \begin{equation}
  \label{boundary_component_Mathai_Quillen_form}
  \psi_{\partial}(\partial_{+}M,\partial_{-}M,g):=\int^{1}_{0}\mathsf{inc}_{s}^{*} \iota_{\partial_{s}} h^{*}\psi(M,g)ds\in \Omega^{m-2}(Q;\pi^{*}\Theta_{M}),
 \end{equation}
 is defined by using the homotopy
 $$
 h:Q\times[0,1]\rightarrow Q\subseteq TM\backslash M\quad
 \text{ given by }\quad
 h_{s}:=s\cdot\mathsf{id}+(1-s)(\varsigma\circ \pi),
 $$
 with $\varsigma$ being the unit vector field in (\ref{definitions_out_in_unit_vector_fields_at_boundary_plus}),
 and $\mathsf{inc}_{s}:Q\rightarrow Q\times [0,1]$, canonical inclusion, and $\iota_{\partial_{s}}$ indicates the contraction with respect to the vector field $\partial_{s}$.
 \end{defi}

  \begin{lem}
 \label{Mathai_Quillen_form_property_2}
  In analogy with (\ref{derivative_Quillen_form_equals_pull_bakc_Euler_form}) and (\ref{Change_Mathai_Quillen_form_wrt_involution}), the relative Mathai--Quillen form from Definition \ref{Relative_Mathai_Quillen_form} satisfies
  \begin{equation}
  \label{Mathai_Quillen_form_1}
  \mathbf{d}\underline{\psi}(\mathbb{M},g)=\pi^{*}\underline{\mathbf{e}}(\mathbb{M},g),
  \end{equation}
  where $\mathbf{d}$ is the differential given in (\ref{equation_differential_in_cone}) and
   \begin{equation}
  \label{Mathai_Quillen_form_flip}
    \underline{\psi}(\mathbb{M},g)=(-1)^{m}\xi^{*}\underline{\psi}({\mathbb{M}^{\prime}},g),
  \end{equation}
  where $\xi$ is the involution in (\ref{involution_at_boundary}). Moreover, if $g_{0}$ and $g_{1}$ be two Riemannian metrics on $M$, then
  \begin{equation}
  \label{Mathai_Quillen_form_metric_dependance}
  \underline{\psi}(\mathbb{M},g_{1})-\underline{\psi}(\mathbb{M},g_{0})=\pi^{*}\underline{\widetilde{\mathbf{e}}}(\mathbb{M},g_{0},g_{1})
  \end{equation}
  modulo $\mathbf{d}(\Omega^{m-2}(TM\backslash M,Q;\pi^{*}\Theta_{M}))$.
 \end{lem}

 \begin{proof}
 First, from (\ref{derivative_Quillen_form_equals_pull_bakc_Euler_form}), it follows
 \begin{equation}
  \label{derivative_h_star_Mathai_Quillen_is_zero}
 dh^{*}\psi(M,g)=h^{*}d\psi(M,g)=h^{*}\pi^{*}\mathbf{e}(M,g)=0,
 \end{equation}
 where the last equality holds since $\pi^{*}\mathbf{e}(M,g)$ being a $m$-form, its pull-back by $h$ to the boundary must vanish for dimensional reasons. Then, by applying the exterior derivative to
 (\ref{boundary_component_Mathai_Quillen_form}), using its naturality with respect to pull-backs, the (Lie) derivative $\frac{d}{ds}=d\circ\iota_{\partial_{s}}-\iota_{\partial_{s}}\circ d$, formula (\ref{derivative_Quillen_form_equals_pull_bakc_Euler_form}), Stokes' Theorem and
  $\mathsf{inc}_{0}^{*}=\mathsf{inc}_{1}^{*}=\mathsf{id}$, we obtain
 \begin{equation}
  \label{derivative_boundary_component_Mathai_Quillen}
 d^{\partial}\psi_{\partial}(\partial_{+}M,\partial_{-}M,g)=i^{*}\psi(M,g)-\pi^{*}\varsigma^{*}\psi(M,g).
 \end{equation}
 Then, using
 $
 \varsigma^{*}\psi(M,g)=\mathbf{e}_{\partial}(\partial_{+}M,\partial_{-}M,g)
 $
 given in \cite[Formula (2.10)]{Bruening-Ma}, together with (\ref{derivative_Quillen_form_equals_pull_bakc_Euler_form}), (\ref{derivative_boundary_component_Mathai_Quillen}) and (\ref{equation_differential_in_cone}),
 the first claim follows. The behavior of the relative Mathai--Quillen form with respect to the involution $\xi$ follows immediately from (\ref{Change_Mathai_Quillen_form_wrt_involution}).

 We now prove (\ref{Mathai_Quillen_form_metric_dependance}). Consider the transgressed Euler form
  $\widetilde{\mathbf{e}}\left(M,g_{0},g_{1}\right)$
 from Definition \ref{definition_of_Chern_Simons_secondary_classes_1} in the Appendix. With (\ref{derivative_Quillen_form_equals_pull_bakc_Euler_form}), which in this case translates as $d\psi(I\times TM, \nabla^{\mathcal{TM}})=\pi^{*}\mathbf{e}(I\times TM,\nabla^{\mathcal{TM}})$, we obtain
 \begin{equation}
 \label{Transgression_inner_component}
 \psi(M,g_{1})-\psi(M,g_{0})=\pi^{*}\widetilde{\mathbf{e}}\left(M,g_{0},g_{1}\right)+d\int^{1}_{0}\mathsf{inc}_{t}^{*}\iota_{\partial t}\psi(I\times TM,\nabla^{\mathcal{TM}})dt
 \end{equation}
 Now, consider the homotopy 
 $$
 \widetilde{h}:I\times I\times\partial M\rightarrow Q\quad\text{ given by }\quad \widetilde{h}_{s,t}:=s\cdot\mathsf{id}+(1-s)(\widetilde{\varsigma}\circ\pi)
 $$
 and remark that
 $$
 \iota_{\partial_{s}}d\widetilde{h}^{*}\psi(I\times M,\nabla^{\mathcal{TM}})=\iota_{\partial_{s}}\widetilde{h}^{*}\pi^{*}\mathbf{e}(I\times M, \nabla^{\mathcal{TM}})=0
 $$
 Analogously, consider the Chern--Simons' form $\widetilde{\mathbf{e}}(\partial_{+}M,\partial_{-}M, g_{0},g_{1})$ from Definition \ref{definition_of_Chern_Simons_secondary_classes_1}. Then, we have
 \begin{equation}
 \label{Transgression_boundary_component}
 \begin{array}{l}
 \psi_{\partial}(\partial_{+}M,\partial_{-}M,g_{1})-\psi_{\partial}(\partial_{+}M,\partial_{-}M,g_{0})\\
 \hspace{1cm}=\int^{1}_{0}\mathsf{inc}_{t}^{*}\iota_{\partial_{t}}\psi(I\times M,\nabla^{\mathcal{TM}})dt-\pi^{*}\widetilde{\mathbf{e}}(\partial_{+}M,\partial_{-}M, g_{0},g_{1})\\
 \hspace{3cm}
 -d\int^{1}_{0}\int^{1}_{0}\mathsf{inc}_{s,t}^{*}\iota_{\partial_{s}}\iota_{\partial_{t}}\widetilde{h}^{*}\psi(I\times M,\nabla^{TM})ds dt.
 \end{array}
 \end{equation}
 Combining (\ref{Transgression_inner_component}) and (\ref{Transgression_boundary_component}), we obtain formula (\ref{Mathai_Quillen_form_metric_dependance}) expressing the dependance of the Mathai--Quillen form on the metric.
 \end{proof}

 Let $X$ be adapted to $\mathbb{M}$ as in Definition \ref{definition_adapted_vector_field_bordism}. Then, we have a smooth map of pairs $X:(M\backslash\mathcal{X})\rightarrow(TM\backslash M,Q)$, where $Q$ is as in Definition \ref{Relative_Mathai_Quillen_form} and hence
 $$
 X^{*}\underline{\psi}(\mathbb{M},g)\in\Omega^{m-1}(M\backslash\mathcal{X},\partial M;\Theta_{M}^{\C}).
 $$
 Moreover the integral $\int_{(M\backslash\mathcal{X})} X^{*}\underline{\psi}(\mathbb{M},g)\wedge \omega$ is absolute convergent for each $\omega\in\Omega^{1}_{c}(M\backslash\mathcal{X};\C)$
 vanishing on a neighborhood of $\mathcal{X}$. 
 \begin{lem}
  \label{Mathai_Quillen_form_and_Index}
  For every smooth function $f$ on $M$, being locally constant on a neighborhood of $\mathcal{X}$, we have
 $$
 (-1)^{m}\int_{(M\backslash\mathcal{X},\partial M)}X^{*}\underline{\psi}(\mathbb{M},g)\wedge df=\int_{(M,\partial M)}\underline{\mathbf{e}}(\mathbb{M},g)f-\sum_{_{x\in\mathcal{X}}}\mathsf{Ind}_{X}(x)f(x) $$
 \end{lem}
 \begin{proof}
 This follows from fully developping the integral 
 $$
 \int_{(M\backslash \mathcal{X}, \partial M)}\mathbf{d}(X^{*}\underline{\psi}(\mathbb{M},g)\wedge f)
 $$
 by using the graded Leibniz formula in (\ref{equation_Leibnit_rule_relative_forms}), the paring (\ref{pairing}), 
 the identity (\ref{Mathai_Quillen_form_1}), Stokes' Theorem and that
 $$
 \lim_{\delta\rightarrow 0}\int_{-\partial\mathbb{B}^{m}(\delta,x)}X^{*}\psi(M,g)=\mathsf{Ind}_{X}(x)
 $$
 for every zero $x\in\mathcal{X}$. The signs conventions when using Stokes' Theorem are taken with the standard convention as in the proof of Lemma \ref{Lemma_main_properties_of_S_in_particular_linearity_independance_and_how_it_acts_on_exact_forms}. 
 \end{proof}

The following Lemma gives a relative version of the Hopf's formula at the same time.
\begin{lem}
 \label{Theorem_Chern_Gauss_Bonnet_Bordism}
 For the bordism $\mathbb{M}$ consider an adapted vector field $X$ as in Definition \ref{definition_adapted_vector_field_bordism}.
 Let $\chi(M,\partial_{-}M)$ be
 the Euler characteristic relative to $\partial_{-}M$. Then, with $\underline{\mathbf{e}}(\mathbb{M},g)$
 the relative Euler form given in Definition \ref{definition_relative_Euler_form} and the pairing in (\ref{pairing}), we have 
 $$
 \begin{array}{c}
\chi(M,\partial_{-}M) = \int_{(M,\partial M)}\underline{\mathbf{e}}(\mathbb{M},g)=\sum_{x\in\mathcal{X}}\mathsf{Ind}_{X}(x)
 \end{array}
 $$
\end{lem}
\begin{proof}
 The first equality is a restatement of the (Chern--Gauss--Bonnet) formula in \cite[Theorem 3.4]{Bruening-Ma2}, see also \cite[Theorem 6.1.14]{Maldonado_thesis}, in terms of the relative form from Definition \ref{definition_relative_Euler_form}. The second equality, directly follows from Lemma \ref{Mathai_Quillen_form_and_Index} using the function $f=1$.
\end{proof}

\subsection{Poincar\'e duality}
 For simplicity, assume $\chi(M,\partial_{-}M)=0$. Consider the spaces of Euler and co-Euler structures on $\mathbb{M}$. The following
 generalizes Proposition 5 in \cite{Burghelea-Haller06b}
 \begin{thm}
 There is a natural isomorphism of affine spaces
 \begin{equation}
  \label{Poincare_Duality_pairing}
  \mathsf{P}:\mathbf{\mathfrak{Eul}}(\mathbb{M};\C)^{*}\rightarrow  \mathbf{\mathfrak{Eul}}(\mathbb{M};\C),
 \end{equation}
 which intertwines the flip map $\nu^{*}$ with $\nu$ and is affine over the Poincar\'e--Lefschetz duality
 \begin{equation}
  \label{Poincare_Lefchtez_Duality_EulerStructures}
 \mathsf{PD}:H^{m-1}(M,\partial M;\Theta_{M}^{\C}) \rightarrow H_{1}(M;\C).
 \end{equation}
 In other words, for every $\underline{\beta}\in H^{m-1}(M,\partial M;\Theta_{M}^{\C})$ and every co-Euler structure $\mathfrak{e}^{*}\in\mathfrak{Eul}(\mathbb{M};\C)$
 we have
 \begin{equation}
  \label{Poincare_Duality_Affine_property_of_P}
  \mathsf{P}(\mathfrak{e}^{*}+\underline{\beta})=\mathsf{P}(\mathfrak{e}^{*})+\mathsf{PD}(\underline{\beta}).
\end{equation}

 \end{thm}
 \begin{proof}
 Let $(g,\underline{\alpha})$ be a pair representing the co-Euler structure $\mathfrak{e}^{*}$ and $\underline{\psi}(\mathbb{M},g)$ the relative Mathai--Quillen form from Definition \ref{Relative_Mathai_Quillen_form}.
 Choose a vector field $X$ which is transverse to the zero section, inward pointing along $\partial_{+}M$ and
 outward pointing along $\partial_{-}M$ and with set of isolated singularities $\mathcal{X}$ in the interior of $M$.
 Since $\mathbf{d}\underline{\alpha}=\underline{\mathbf{e}}(\mathbb{M},g)$ and Lemma \ref{Mathai_Quillen_form_property_2}, the relative form $X^{*}\underline{\psi}(\mathbb{M},g)-\underline{\alpha}$ is closed
 and therefore defines a relative cohomology class in $H^{m-1}(M\backslash\mathcal{X},\partial M;\Theta_{M})$. Now, we identify the relative cohomology class
 $[X^{*}\underline{\psi}(\mathbb{M},g)-\underline{\alpha}]$ to its dual Poincar\'e--Lefschetz class $[c]\in H_{1}(M, \mathcal{X};\C)$ represented by a singular 1-chain $c$, by the requirement
 \begin{equation}
  \label{Poincare_Duality_pairing_Co_Euler_Euler_structures}
  \int_{(M\backslash\mathcal{X},\partial M)} \left(X^{*} \underline{\psi}(\mathbb{M},g)-\underline{\alpha}\right)\wedge\omega=\int_{c}\omega
  \end{equation}
 to hold for all closed 1-forms $\omega\in\Omega^{1}(M;\C)$ compactly supported on $M\backslash\mathcal{X}$.
 Moreover, it is possible to choose a singular $1$-chain $c$ which is an Euler chain, i.e. $\partial c=\mathbf{e}(X)$ with $\mathbf{e}(X)$ is the $0$-chain from (\ref{relative_Euler_chain}). 
 Indeed, in the case $\chi(M,\partial_{-}M)=0$, this follows by setting $\omega = df$ for an arbitrary smooth function $f$, developping the left hand side of the identity in (\ref{Poincare_Duality_pairing_Co_Euler_Euler_structures}) with
 (Gauss--Bonnet Theorem in) Lemma \ref{Theorem_Chern_Gauss_Bonnet_Bordism} and using Stokes' Theorem on the right hand side of the identity in (\ref{Poincare_Duality_pairing_Co_Euler_Euler_structures})

 The assignment
 $
   P:(g,\underline{\alpha})\mapsto (X,c)
 $
 specified by the condition (\ref{Poincare_Duality_pairing_Co_Euler_Euler_structures})
 induces the map (\ref{Poincare_Duality_pairing}). This follows from 
 (\ref{defining_relation_one_chain}), formula (\ref{integration_over_chern-simons_cochain}), Lemma \ref{Mathai_Quillen_form_property_2} and Proposition \ref{Proposition_regularization_function_S_both_boundaries}, and
 using the same strategy as that in the situation of closed manifolds, see
 \cite[Lemma 2 and (19)]{Burghelea-Haller06b}. That is, $P$ does depend on neither representative of Euler structure, co-Euler structure and cohomology clasess in $H^{1}(M;\C)$ and one obtains the pairing
 $$
 \mathbb{T}:\mathfrak{Eul}^{*}(\mathbb{M};\C)\times\mathfrak{Eul}(\mathbb{M};\C)\rightarrow H_{1}(M;\C),
 $$
 with the property 
 \begin{equation}
  \label{pairing_coEuler-Euler_homology}
 \mathbb{T}(\mathfrak{e}^{*}+\underline{\beta},\mathfrak{e}+\sigma)=\mathbb{T}(\mathfrak{e}^{*},\mathfrak{e}) - \sigma + \mathsf{PD}(\underline{\beta})
 \end{equation}
 for every $\mathfrak{e}^{*}\in\mathfrak{Eul}^{*}(\mathbb{M};\C)$, $\mathfrak{e}\in\mathfrak{Eul}(\mathbb{M};\C)$, $\sigma\in H_{1}(M;\C)$ and relative form $\underline{\beta}\in H^{m-1}(M,\partial M;\Theta_{M}^{\C})$.
 Using (\ref{pairing_coEuler-Euler_homology}) and that $\mathfrak{Eul}^{*}(\mathbb{M};\C)$ and $\mathfrak{Eul}(\mathbb{M};\C)$ are affine spaces over relative cohomology and homology groups respectively, one obtains that $\mathsf{P}$ is affine over the homomorphism $\mathsf{PD}$ expressing the Poincar\'e--Lefschetz duality in (\ref{Poincare_Lefchtez_Duality_EulerStructures}),
 and hence formula (\ref{Poincare_Duality_Affine_property_of_P}) holds.
 Since $\mathsf{PD}$ is an isomorphism, $\mathsf{P}$ is so. 
 Finally because of the properties of the relative Mathai--Quillen form and definition of the involution 
 $\nu$, it is clear that $\mathsf{P}$ intertwines the flip maps $\nu^{*}$ and $\nu$ on the spaces of co-Euler and Euler structures respectively.

 \end{proof}

\section{Co-Euler structures and the complex-valued analytic torsion}
\label{Section_Burghelea--Haller analytic torsion on bordisms}
\index{generalized complex-valued analytic torsion}
In this section, we extend \cite[Theorem 4.2]{Burghelea-Haller} to the situation of a  bordism $\mathbb{M}$.
We refer the reader to \cite{Maldonado} for details, since we use the definitions, notation and results therein.

Let $E$ be a complex flat vector bundle over $M$.
Assume $E$ is endowed with a fiber-wise non-degenerate symmetric bilinear form $b$. Consider
the bilinear Laplacian $\Delta_{E,g,b}:=d_{E}d_{E,g,b}^{\sharp}+d_{E,g,b}^{\sharp}d_{E}$ acting on smooth $E$-valued forms satisfying absolute boundary conditions on $\partial_{+}M$ and
relative boundary conditions on $\partial_{-}M$, see \cite{Maldonado}.

Consider $[\tau(0)]^{E,g,b}_{\mathbb{M}}$ the bilinear form induced in the determinant line $\det H^{*}\left(M,\partial_{-}M;E\right)$,
by the restriction of $b$ to $0$-generalized eigenspace of $\Delta_{E,g,b}$, with the use of a Hodge--de-Rham theorem and the Knudson--Munford isomorphism,
see \cite{Knudson-Mumford} and \cite{Maldonado}.
Then,
the \textit{complex-valued Ray--Singer torsion} \index{complex-valued Ray--Singer torsion} is the bilinear form on $\det H(M,\partial_{-}M;E)$ defined by

\begin{equation}
 \label{Definition_complex_valued_Ray_singer_analytic_torsion}
 [\tau^{\mathsf{RS}}]^{E,g,b}_{\mathbb{M}}:=[\tau(0)]^{E,g,b}_{\mathbb{M}}\cdot\prod_{q}\left({\det}^{\prime}\left(\Delta_{E,g,b,q}\right)\right)^{(-1)^{q}q},
\end{equation}
where ${\det}^{\prime}\left(\Delta_{E,g,b,q}\right)$  is the $\zeta$-regularized determinant of
$\Delta_{E,g,b,q}$ defined as
$$
{\det}^{\prime}\left(\Delta_{\mathcal{B},q}\right):=
\exp(-\left.\frac{\partial}{\partial s}\right|_{s=0}\Tr((\Delta_{E,g,b,q}^{\mathsf{c}})^{-s}))
$$
with
$$
\Delta_{E,g,b,q}^{\mathsf{c}}:=\left.\Delta_{E,g,b,q}\right|_{\left.{\Omega^{q}_{\Delta_{E,g,b,q}}(M;E)(0)^{\mathsf{c}}}\right.|_{\mathcal{B}}}
$$
being the restriction of $\Delta_{E,g,b}$
to the space
of smooth differential forms of degree $q$
which are not in $0$-generalized eigenspace of $\Delta_{E,g,b}$ but satisfy the boundary conditions above.

The \textit{generalized} complex-valued Ray--Singer torsion on closed manifolds was constructed in
in \cite[Theorem 4.2]{Burghelea-Haller}, by
adding appropriate
correction terms to the complex-valued torsion in order to
cancel out the infinitesimal variation to the complex-valued analytic torsion. These correction terms were introduced using co-Euler structures,
once the anomaly formulas for the torsion were computed.
The procedure in the situation on a compact bordism is carried out in a similar fashion. In fact,
the required correction terms are constructed by using this time co-Euler structures on compact bordisms, see Section \ref{section_Co_Euler_structures}, and the anomaly formulas
in \cite[Theorem 3]{Maldonado}.

 \begin{thm}
 \label{definition_Burghele_Haller_torsion}
Let $\mathbb{M}$ be a bordism with Riemannian metric $g$. Assume $\chi(M,\partial_{-}M)=0$. Let
$\mathfrak{e}^{*}\in\mathbf{\mathfrak{Eul}}^{*}(\mathbb{M};\C)$
a the co-Euler structure (without base point), see Section \ref{Section_co-Euler Structures without base point}.
Let $E$ be a complex flat vector bundle over $M$, with flat connection $\nabla^{E}$. Assume
$E$ is endowed with a complex non-degenerate symmetric
bilinear form $b$.
Then,
\begin{equation}
 \label{equation_Burghelea_Haller_torsion_both_boundaries_chi_nul}
\begin{array}{c}
[\tau]^{E,\mathfrak{e}^{*},[b]}_{\mathbb{M}}:=[\tau^{\mathsf{RS}}]^{E,g,b}_{\mathbb{M}}\cdot
 e^{\left({2\int_{(M,\partial M)}\underline{\alpha}\wedge\omega(E,b)-
 \mathsf{rank}(E)\int_{\partial M}B(\partial_{+}M,\partial_{-}M,g)}\right)}\\
\end{array}
\end{equation}
where
\begin{itemize}
 \item $[\tau^{\mathsf{RS}}]^{E,g,b}_{\mathbb{M}}$ is the torsion
on $\mathbb{M}$ in (\ref{Definition_complex_valued_Ray_singer_analytic_torsion}),
 \item  $(g,\underline{\alpha})$ is a representative of the co-Euler structure
$\mathfrak{e}^{*}$,
 \item $B(\partial_{+}M,\partial_{-}M,g)$ is the characteristic form from Definition \ref{definitions_of_modified_e_partial_B_forms_on_bordisms},
 \item $[b]$ indicates the homotopy class of $b$,
 \item $\omega(E,b)$ is the Kamber--Tondeur form for $\nabla^{E}$ and $b$ in (\ref{Kamber-Tondeur_form}),
\end{itemize}
is well defined as bilinear form
on $\det(H(M,\partial_{-}M))$
\end{thm}

\begin{proof}
We have to prove that $[\tau]^{E,\mathfrak{e}^{*},[b]}_{\mathbb{M}}$
is independent of the choice of representatives for the co-Euler structure and it depends on $\nabla^{E}$ and the homotopy class $[b]$ of $b$ only.
For $\{(g_{w},\underline{\alpha}_{w})\}_{w\in U}$ a real one-parameter smooth path of Riemannian metrics $g_{w}$ on $M$ relative forms $\underline{\alpha}_{w}\in\Omega^{m-1}(M,\partial M;\Theta_{M}^{\C})$ representing the \textit{same} co-Euler structure $\mathfrak{e}^{*}\in\mathbf{\mathfrak{Eul}}^{*}(\mathbb{M};\C)$ and
$\{b_{w}\}$ a real one-parameter smooth path  of
non-degenerate symmetric bilinear forms on $E$, consider the family
$[\tau]^{E,(g,\underline{\alpha}_{w}),b_{w}}_{\mathbb{M}}$ of bilinear forms.

We claim that $[\tau]^{E,(g,\underline{\alpha}_{w}),b_{w}}_{\mathbb{M}}$ is independent of the parameter $w$, or equivalently its corresponding logarithmic derivative vanishes.
To prove the claim, fix $u\in U$, consider the complex number
$
[\tau]^{E,(g,\underline{\alpha}_{w}),b_{w}}_{\mathbb{M}}\slash[\tau]^{E,(g,\underline{\alpha}_{w}),b_{u}}_{\mathbb{M}}
$
and remark that its logarithm derivative
with respect to $w$, is the sum of two contributions:
\begin{enumerate}
 \item[(a)] the logarithmic derivative w.r.t. $w$ of the exponential depending on the co-Euler structures:
  $$
\exp\left({2\int_{(M,\partial M)}\underline{\alpha_{w}}\wedge\omega(E,b_{w})-
 \mathsf{rank}(E)\int_{\partial M}B(\partial_{+}M,\partial_{-}M,g_{w})}\right)
  $$
 \item[(b)] the logarithmic derivative w.r.t. $w$
   of
   $
   [\tau]^{E,g,b_{w}}_{\mathbb{M}}\slash[\tau]^{E,g,b_{u}}_{\mathbb{M}},
   $
   which corresponds to the anomaly formulas for the complex-valued Ray--Singer torsion.
\end{enumerate}
 The logarithmic derivative in (b) has been computed
 \cite[Theorem 2]{Maldonado}
 in terms of the characteristic forms, as defined by Br\"uning and Ma,
 \begin{equation}
  \label{proof_final_formula_1}
  B(\partial M,g_{w}),\quad\mathbf{e}(M,g_{w}),\quad\mathbf{e_{b}}(\partial M,g_{w})\quad\text{and}\quad
 \widetilde{\mathbf{e}}_{\mathbf{b}}(\partial M,g_{u},g_{w}).
 \end{equation}
 The logarithmic derivative (a), computed in Proposition
 \ref{Lemma_variation_co-Euler_structures_on_man_with_boundary_without_base_point}, expresses the variation of the representatives of the co-Euler structures, see
 Definition \ref{definition_Co-EulerStructure_without_base_point}, with respect to smooth variations of $w$. The corresponding formulas in
 (\ref{equation_variation_formula_with_out_base_point_1}) and (\ref{equation_variation_formula_with_out_base_point_2})
 from Proposition \ref{Lemma_variation_co-Euler_structures_on_man_with_boundary_without_base_point}
 are written (see Definitions \ref{definitions_of_modified_e_partial_B_forms_on_bordisms} and \ref{definitions_of_modified_e_partial_B_forms_on_bordisms},
 and Definition \ref{definition_relative_Euler_form}, (\ref{definition_Chern_Simons_relative_form_on_bordism}))
 in terms of the characteristic forms
 \begin{equation}
  \label{proof_final_formula_2}
 \begin{array}{rcl}
 B(\partial_{+}M,\partial_{-}M,g_{w})&:=&B_{\varsigma}(\partial M, g_{w})\\
 \underline{\mathbf{e}}(\mathbb{M},g_{w})&:=&(\mathbf{e}(M,g_{w}),\mathbf{e}_{\partial}(\partial_{+}M,\partial_{-}M,g_{w})),\\
 \underline{\mathbf{\widetilde{e}}}(\mathbb{M},g_{u},g_{w})
&:=&\left(\widetilde{\mathbf{e}}\left(M,g_{u},g_{w}\right),
-\left.\widetilde{\mathbf{e}}_{\partial}\right.\left(\partial_{+}M,\partial_{-}M,g_{u},g_{w}\right)\right)
\end{array}
 \end{equation}
where $\varsigma$ is the unit vector field at the boundary defined in (\ref{definitions_out_in_unit_vector_fields_at_boundary_plus}).
But the construction of the forms in (\ref{proof_final_formula_2}) is compatible with the forms from Br\"uning and Ma in (\ref{proof_final_formula_1}). More precisely,
\begin{equation}
  \label{proof_final_formula_3}
 \begin{array}{rcl}
 B(\partial_{+}M,\partial_{-}M,g_{w})|_{\partial_{\pm}M}&=&(\pm 1)^{m-1}B(\partial M,g_{w})|_{\partial_{\pm}M}\\
 \mathbf{e}_{\partial}(\partial_{+}M,\partial_{-}M,g_{w})|_{\partial_{\pm}M}&=&(\pm 1)^{m}\mathbf{e_{b}}(\partial M,g_{w})|_{\partial_{\pm}M}\\
 \left.\widetilde{\mathbf{e}}_{\partial}\right.\left(\partial_{+}M,\partial_{-}M,g_{u},g_{w}\right)|_{\partial_{\pm}M}&=&(\pm 1)^{m}\widetilde{\mathbf{e}}_{\mathbf{b}}(\partial M,g_{u},g_{w})|_{\partial_{\pm}M}
\end{array}
 \end{equation}
 see also Lemma \ref{Lemma_duality_e_partial_on_bordisms}.
 Therefore, with (\ref{proof_final_formula_3}), the contribution from (a) and (b) are the same up to $-1$ factor. The proof is complete.
\end{proof}

\subsection{Without conditions on $\chi(M,\partial_{\pm}M)$}
\label{section_generalized_torsion_with_co-Euler Structures with base point}

Let $\mathbb{M}$ be a  bordism and
$E$ a complex flat vector bundle over $M$ with flat connection $\nabla^{E}$. We
assume it is endowed with a complex non-degenerate symmetric
bilinear form $b$ and $\omega(E,b)$ the corresponding closed 1-form of Kamber---Tondeur, see (\ref{Kamber-Tondeur_form}).
For
$x_{0}\in\mathbf{int}(M)$, let
$\mathfrak{e}^{*}_{x_{0}}\in\mathbf{\mathfrak{Eul}}_{x_{0}}^{*}(\mathbb{M};\C)$
be a co-Euler structures based at $x_{0}$, see Definition \ref{definition_Set_of_coeuler_structures_with_base_point},
represented by $(g,\underline{\alpha})$, where $\underline{\alpha}:=(\alpha,\alpha_{\partial})$ is a relative form with
$\alpha\in\Omega^{m-1}(\dot{M};\Theta_{M}^{\C})$ and $\dot{M}:=M\backslash\{x_{0}\}.$

Let $b_{(\det E_{x_{0}})^{-\chi(M,\partial_{-}M)}}$ be the induced bilinear form
on $(\det E_{x_{0}})^{-\chi(M,\partial_{-}M)}$. 
Consider $\tau^{\mathsf{RS}}_{E,g,b}$ the complex-valued Ray--Singer torsion on $\mathbb{M}$,
 and
$\mathcal{S}$
the function regularizing $\int_{(M,\partial M)}$
studied in Proposition \ref{Proposition_regularization_function_S_both_boundaries}.

\begin{thm}
\label{Theorem_Burghelea_Haller_torsion_both_boundaries_chi_general_anomaly_formula}
The formula
\begin{equation}
 \label{equation_Burghelea_Haller_torsion_both_boundaries_chi_general}
\begin{array}{c}
 \tau^{\mathsf{an}}_{E,\mathfrak{e}^{*}_{x_{0}},[b]}:=\tau^{\mathsf{RS}}_{E,g,b }\cdot
 \mathsf{e}^{2\mathcal{S}(\underline{\alpha},\omega(E,b))-
 \mathsf{rank}(E)\int_{\partial M}B(\partial_{+}M,\partial_{-}M,g)}\otimes
 b_{(\det E_{x_{0}})^{-\chi(M,\partial_{-}M)}},\\
\end{array}
\end{equation}
defines a bilinear form on $\det(H(M,\partial_{-}M))\otimes (\det E_{x_{0}})^{-\chi(M,\partial_{-}M)}$,
which is independent of the choice of representative for the co-Euler structure and depends
on the connection and the homotopy class $[b]$ of $b$ only.
\end{thm}

\begin{proof}
On the one hand, if $b$ is fixed and we only look at changes of the metric, then
the variation of $\tau^{\mathsf{an}}_{E,(g,(\underline{\alpha},\theta)),b}$ with respect to the metric compensates
the variation of the function $\mathcal{S}(\underline{\alpha},\omega(E,b))$, which is explicitly given by formula
(\ref{Lemma_variation_of_S_b_constant}) in Proposition \ref{Proposition_regularization_function_S_both_boundaries}.
On the other hand, when $g$ and $\mathfrak{e}^{*}_{x_{0}}$ are kept constant
and we allow $b$ to smoothly change from $b_{1}$ to $b_{2}$, then the variation of the Kamber--Tondeur form is given by
$$
\begin{array}{rcl}
\scriptstyle\omega(E,b_{2})-\omega(E,b_{1})=
\scriptstyle-\frac{1}{2}\det((b_{1}^{-1}b_{2})^{-1})d \det\left(b_{1}^{-1}b_{2}\right)
=-\frac{1}{2}d\log\det((b_{1}^{-1}b_{2})^{-1}),
\end{array}
$$
where the last equality holds, since $b_2$ and $b_1$ are homotopic and therefore the function
$$\det((b_{1}^{-1}b_{2})^{-1}):M\rightarrow\C\backslash\{0\},$$
is homotopic to the constant function 1, which in turn
allows to find a function $$\log\det((b_{1}^{-1}b_{2})^{-1}):M\rightarrow\C,$$ with
$$d\log\det((b_{1}^{-1}b_{2})^{-1})=
\det((b_{1}^{-1}b_{2})^{-1})d \det(b_{1}^{-1}b_{2}).$$
This, with
$f= \Tr((b_{1}^{-1}b_{2})^{-1})$ and Lemma \ref{Lemma_main_properties_of_S_in_particular_linearity_independance_and_how_it_acts_on_exact_forms},
implies that
$$
\begin{array}{l}
\scriptstyle 2\mathcal{S}_{f}(\underline{\alpha},\omega(E,b_{2}))-2\mathcal{S}_{f}(\underline{\alpha},\omega(E,b_{1}))=
2\mathcal{S}_{f}\left(\underline{\alpha},d\log\det\left(\left(b_{1}^{-1}b_{2}\right)^{-1}\right)\right)\\ \\
\scriptstyle\hspace{1.5cm}=-(-1)^{m} \int_{(M,\partial M)}\underline{\mathbf{e}}(\mathbb{M},g)\log\det\left(\left(b_{1}^{-1}b_{2}\right)^{-1}\right)\\
\scriptstyle\hspace{6cm}+\log\det\left(\left(b_{1}^{-1}b_{2}\right)^{-1}\right)(x_{0})\chi(M,\partial_{-}M)\\ \\
\scriptstyle\hspace{1.5cm}=-(-1)^{m} \int_{(M,\partial M)}\underline{\mathbf{e}}(\mathbb{M},g) \Tr\left(\left(b_{1}^{-1}b_{2}\right)^{-1}\right)\\
\scriptstyle\hspace{6cm}+
\Tr\left(\left(b_{1}^{-1}b_{2}\right)^{-1}\right)(x_{0})\chi(M,\partial_{-}M),
\end{array}
$$
where the additional term $ \Tr((b_{1}^{-1}b_{2})^{-1})(x_{0})\chi(M,\partial_{-}M)$ cancels the variation
of the induced bilinear form on $(\det E_{x_{0}})^{-\chi(M,\partial_{-}M)}$ given by
$$
\left({b_{1}}_{(\det E_{x_{0}})^{-\chi(M,\partial_{-}M)}}\right)^{-1}{b_{2}}_{(\det E_{x_{0}})^{-\chi(M,\partial_{-}M)}}=\det(b_{1}^{-1}b_{2})^{-\chi(M,\partial_{-}M)}.
$$

\end{proof}

\subsection{Complex-valued analytic torsion and Poincar\'e duality}
\index{generalized complex-valued analytic torsion and Poincar\'e duality}
Let us consider the bordism $\mathbb{M}^{\prime}$ in (\ref{bordism}) dual to $\mathbb{M}$,
$E^{\prime}$ the dual complex vector bundle of $E$ endowed with the corresponding dual connection and $b^{\prime}$ the
non-degenerate symmetric bilinear form dual to $b$ on $E$. 
By Poincar\'e--Lefschetz duality, we have
$$
H^{p}(M,\partial_{+}M;E^{\prime}\otimes\Theta_{M})\cong {H^{m-p}(M,\partial_{-}M;E)}^{\prime}
$$
and hence there is a canonic isomorphism of determinant line bundles
\begin{equation}
 \label{eqation_equality_of_determinant_lines}
 \det\left(H\left(M,\partial_{+}M;E^{\prime}\otimes\Theta_{M}\right)\right)\cong
 \left(\det\left(H(M,\partial_{-}M;E)\right)\right)^{(-1)^{m+1}},
\end{equation}
see for instance \cite{Knudson-Mumford}, \cite{Milnor_Cobordism} and \cite{Milnor_Whitehead_Torsion}.
The bilinear Laplacians $\Delta_{E,g,b,q}$ and
$\Delta_{E^{\prime}\otimes\Theta_{M},g,b^{\prime},m-q}$, as well as the corresponding boundary conditions
are intertwined by the isomorphism
$\star_{g}\otimes b:\Omega^{q}(M;E)\rightarrow\Omega^{m-q}(M;E^{\prime}\otimes\Theta_{M}).$ This implies that
their $\mathsf{L}^{2}$-realizations of $\Delta_{E,g,b,q}$ and
$\Delta_{E,g,b,m-q}^{\prime}$ are isospectral, and therefore
\begin{equation}
\label{equation_Laplacian_isospectral_then_theri_determinats_are_the_same}
{\det}^{\prime}(\Delta_{E,g,b,q})=
{\det}^{\prime}({\Delta_{E^{\prime}\otimes\Theta_{M},g,b^{\prime},m-q}}).
\end{equation}

By definition of the torsion in (\ref{equation_Burghelea_Haller_torsion_both_boundaries_chi_nul}), the isomorphism in
(\ref{equation_Laplacian_isospectral_then_theri_determinats_are_the_same}),
the identity in (\ref{eqation_equality_of_determinant_lines}), the formula
$
\Pi_{q}\left({\det}^{\prime}(\Delta_{E,g,b,q})\right)^{(-1)^{q}}=1
$
see \cite{Maldonado_thesis},
the relation between the forms $B(\partial_{+}M,\partial_{-}M,g)$ and $B(\partial_{-}M,\partial_{+}M,g)$
from Lemma \ref{Lemma_duality_e_partial_on_bordisms},
and
\begin{equation}
\label{equation_relation_Kamber_Tondeur_Form_dual}
\omega(E^{\prime}\otimes\Theta_{M},b^{\prime})=-\omega(E,b),
\end{equation}
see \cite[Section 2.4]{Burghelea-Haller}, we obtain
\begin{equation}
 \label{equation_Burghelea_Haller_torsion_both_boundaries_and_Poincare_duality}
 [\tau]^{E^{\prime}\otimes\Theta_{M},\nu^{*}(\mathfrak{e}^{*}),[b^{\prime}]}_{\mathbb{M}^{\prime}} =
 \left([\tau]^{E,\mathfrak{e}^{*},[b]}_{\mathbb{M}}\right)^{(-1)^{m+1}},
\end{equation}
where $\nu^{*}:\mathfrak{Eul}^{*}(\mathbb{M};\C)\rightarrow\mathfrak{Eul}^{*}({\mathbb{M}^{\prime}};\C)$
is the map in (\ref{equation_flip_map_co-Euler_structures}), intertwining the corresponding
co-Euler structures. The formula in (\ref{equation_Burghelea_Haller_torsion_both_boundaries_and_Poincare_duality})
exhibits the behavior of generalized complex-valued torsion on the bordism
$\mathbb{M}$ under Poincar\'e--Lefschetz duality, generalizing this situation
in the case without boundary, see \cite[(31)]{Burghelea-Haller}.


 \section{Appendix}
  \label{section_Appendix}
  In this section, for the reader's convenience, we stay close to the notation in \cite{Bruening-Ma} (see also \cite[Chapter 3]{Bismut-Zhang}).
  \subsection{The Berezin integral and Pfaffian}
  \label{section_The Berezin integral and Pfaffian}
  For $A$ and $B$ two unital $\Z_{2}$-graded algebras, with respective unities
  $1_{A}$ and $1_{B}$, we consider their $\Z_{2}$-\textit{graded tensor product} denoted by
  $A\widehat{\otimes}B$. The map $w\mapsto w\widehat{\otimes} 1_{B}$ provides a canonical isomorphism
  between $A$ and the subalgebra $A\widehat{\otimes} 1_{B}\subset A\widehat{\otimes} B$, whereas with the map
  $w\mapsto \widehat{w}:=1_{A}\widehat{\otimes} w$ we canonically identify $B$ with
  the subalgebra $\widehat{B}:=1_{A}\widehat{\otimes}B\subset A\widehat{\otimes} B$. As $\Z_{2}$-graded algebras, one has
  $A\widehat{\otimes}\widehat{B}\cong A\widehat{\otimes}B$.

  Let $W$ and $V$ be finite dimensional vector
  spaces of dimension $n$ and $l$ respectively, with $W^{\prime}$ and $V^{\prime}$ their corresponding dual spaces. We
  denote by $\Theta_{W}$ the orientation line of $W$. Assume $W$ is endowed with a Hermitian product $\langle\cdot,\cdot\rangle$,
  fix $\{w_{i}\}_{i=1}^{n}$ an orthonormal basis of $W$ and use the metric to fix $\{w^{i}\}_{i=1}^{n}$ the corresponding dual basis in $W^{\prime}$.
  Then, each antisymmetric endomorphism $K$ of $W$ can be uniquely identified with
  the section $\mathbf{K}$ of $\widehat{\Lambda(W^{\prime})}$ given by
  $$
  \mathbf{K}:=\frac{1}{2}\sum_{1\leq i,j\leq n}\langle w_{i},K w_{j}\rangle\widehat{w^{i}}\wedge\widehat{w^{j}}.
  $$

  The \textit{Berezin integral} $$\int^{B}:\Lambda V^{\prime}\widehat{\otimes}\widehat{\Lambda(W^{\prime})}\rightarrow\Lambda V^{\prime}\otimes\Theta_{W}$$
  is the linear map given by $\alpha\widehat{\otimes}\widehat{\beta}\mapsto C_{B}\beta_{g,b}(w_{1},\ldots,w_{n})$,
  with constant
  $C_{B}:=(-1)^{n(n+1)/2}\pi^{-n/2}$.
  Then, $\mathbf{Pf}\left(K/2\pi\right)$, the \textit{Pfaffian} of $K/2\pi$, is defined by
  $$
  \mathbf{Pf}\left(K/2\pi\right):=\int^{B}\exp(\mathbf{K}/2\pi).
  $$
  Remark that $\mathbf{Pf}\left(K/2\pi\right)=0$, if $n$ is odd.
  By standard fiber-wise considerations the map $\mathbf{Pf}$ is extended for
  vector bundles over $M$.

  \subsection{Certain characteristic forms on the boundary}
  Let $M$ be a $m$-dimensional compact Riemannian manifold with boundary $\partial M$ and denote by $i:\partial M\hookrightarrow M$ the canonical embedding.
  We denote by $g:=g^{TM}$ (resp.  $g^{\partial}:=g^{T\partial M}$) the Riemannian metric on $TM$
  (resp. on $T\partial M$ and induced by $g$), by $\nabla$ (resp. $\nabla^{\partial}$) the
  corresponding Levi-Civita connection and by $\mathsf{R}^{{TM}}$ (resp. $\mathsf{R}^{{T\partial M}}$) its curvature.
  Let $\{e_{i}\}_{i=1}^{m}$ be an orthonormal frame of $TM$ with the property that near the
  boundary, $e_{m}=\varsigma_{\mathsf{in}}$, i.e., the inward pointing unit normal vector field on the boundary.
  The corresponding induced orthonormal local frame on $T\partial M$ will be denoted by $\{e_{\alpha}\}_{\alpha=1}^{m-1}$. As usual, the metric is used to fix
  $\{e^{i}\}_{i=1}^{m}$ (resp. $\{e^{\alpha}\}_{\alpha=1}^{m-1}$) the corresponding dual frame of $T^{*}M$ (resp. $T^{*}\partial M$).

  With the notation in Appendix \ref{section_The Berezin integral and Pfaffian}, a smooth section $w$ of $\Lambda T^{*}M$
  is identified with the section $w\widehat{\otimes}1$ of $\Lambda T^{*}M\widehat{\otimes}\Lambda T^{*}M$, whereas
  $\widehat{w}$ denotes the corresponding section $1\widehat{\otimes}w$ of $\Lambda T^{*}M\widehat{\otimes}\Lambda T^{*}M$.

  Here, the Berezin integrals
  $
    \int^{B_{M}}:\Lambda T^{*}M\widehat{\otimes}\widehat{\Lambda T^{*}M}\rightarrow\Lambda T^{*}M\otimes\Theta_{M}
  $
  and
  $
    \int^{B_{\partial M}}:\Lambda T^{*}\partial M\widehat{\otimes}\widehat{\Lambda(T^{*}\partial M)}\rightarrow\Lambda T^{*}\partial M\otimes\Theta_{\partial M}
  $
  can be compared under the given convention for the induced orientation bundle on the boundary, see
  Section \ref{section_background}.

  The curvature
 $\mathsf{R}^{{TM}}$ associated to $\nabla$, considered as a smooth section of
 $\Lambda^{2}(T^{*}M)\widehat{\otimes}\widehat{\Lambda^{2}(T^{*}M})\rightarrow M,$
 can be expanded in terms of the frame above as
$$
\begin{array}{rcll}
\scriptstyle\mathbf{R}^{{TM}}&\scriptstyle:=&\scriptstyle\frac{1}{2}\sum_{1\leq k,l\leq m}g^{{TM}}\left(e_{k},\mathsf{R}^{{TM}}e_{l}\right)\widehat{e^{k}}\wedge\widehat{e^{l}}
&\scriptstyle\in\quad\Gamma({M};\Lambda^{2}(T^{*}{M})\widehat{\otimes}\widehat{\Lambda^{2}({T^{*}M})})\\
\end{array}
$$
In the same way, consider the forms
\begin{equation}
\label{equation_definition_of_S_both_boundaries}
\begin{array}{rcll}
\scriptstyle i^{*}\mathbf{R}^{{TM}}&\scriptstyle:=&\scriptstyle\frac{1}{2}\sum\limits_{1\leq k,l\leq m}
g^{{TM}}(e_{k},i^{*}\mathsf{R}^{{TM}}e_{l})\widehat{e^{k}}\wedge\widehat{e^{l}}
&\scriptstyle\in\quad\Gamma({\partial M};\Lambda^{2}(T^{*}{\partial M})\widehat{\otimes}\widehat{\Lambda^{2}({T^{*}M})}),\\

\scriptstyle\left.\mathbf{R}^{{TM}}\right|_{\partial M}&\scriptstyle:=&\scriptstyle\frac{1}{2}\sum\limits_{1\leq \alpha,\beta\leq m-1}
g^{{TM}}(e_{\alpha},i^{*}\mathsf{R}^{{TM}}e_{\beta})\widehat{e^{\alpha}}\wedge\widehat{e^{\beta}}
&\scriptstyle\in\quad\Gamma({\partial M};\Lambda^{2}(T^{*}{\partial M})\widehat{\otimes}\widehat{\Lambda^{2}({T^{*}(\partial M)}))}),\\

\scriptstyle\mathbf{R}^{{T \partial M}}
&\scriptstyle:=&\scriptstyle\frac{1}{2}\sum\limits_{1\leq \alpha,\beta\leq m-1}
g^{{T \partial M}}(e_{\alpha},\mathsf{R}^{{T\partial M}}e_{\beta})\widehat{e^{\alpha}}\wedge\widehat{e^{\beta}}
&\scriptstyle\in\quad\Gamma({\partial M};\Lambda^{2}(T^{*}{\partial M})\widehat{\otimes}\widehat{\Lambda^{2}({T^{*}(\partial M)}))}),\\

\scriptstyle\mathbf{S}_{\varsigma}&\scriptstyle:=&
\scriptstyle\frac{1}{2}\sum\limits_{\beta=1}^{m-1}
g^{{TM}}((i^{*}\nabla^{{TM}})\varsigma,e_{\beta})\widehat{e^{\beta}}
&\scriptstyle\in\quad\Gamma({\partial M};T^{*}{\partial M}\widehat{\otimes}\widehat{\Lambda^{1}({T^{*}(\partial M)})})\\

\scriptstyle\mathbf{S}_{}&\scriptstyle:=&
\scriptstyle\frac{1}{2}\sum\limits_{\beta=1}^{m-1}
g^{{TM}}((i^{*}\nabla^{{TM}})\varsigma_{\mathsf{in}},e_{\beta})\widehat{e^{\beta}}
&\scriptstyle\in\quad\Gamma({\partial M};T^{*}{\partial M}\widehat{\otimes}\widehat{\Lambda^{1}({T^{*}(\partial M)})})\\

\end{array}
\end{equation}
to define
\begin{equation}
\label{label_in_article_differential_forms_boundary}
\begin{array}{rcl}
\mathbf{e}({M},\nabla^{{T M}})&:=&\int^{B_{M}}\exp\left(-\frac{1}{2}\mathbf{R}^{{T M}}\right),\\

\mathbf{e}({\partial M},\nabla^{{T \partial M}})&:=&\int^{B_{\partial M}}\exp\left(-\frac{1}{2}
\mathbf{R}^{{T \partial M}}\right),\\

\mathbf{e}_{\mathbf{b},\varsigma}({\partial M}, \nabla^{{T M}})&:=&
(-1)^{m-1}\int^{B_{\partial M}}\exp\left(-\frac{1}{2}(\mathbf{R}^{{T M}}|_{\partial  M})\right)\sum_{k=0}^{\infty}\frac{\mathbf{S}^{k}_{\varsigma}}{2\Gamma(\frac{k}{2}+1)},\\

B_{\varsigma}({\partial M}, \nabla^{{T M}})  &:=&
-\int^{1}_{0}\frac{du}{u}\int^{B_{\partial M}}\exp\left(-\frac{1}{2}\mathbf{R}^{{T \partial M}}-u^{2}\mathbf{S}_{\varsigma}^{2}\right)\sum_{k=1}^{\infty}\frac{\left(u\mathbf{S}_{\varsigma}\right)^{k}}{2\Gamma(\frac{k}{2}+1)},\\

B({\partial M}, \nabla^{{T M}})  &:=&
-\int^{1}_{0}\frac{du}{u}\int^{B_{\partial M}}\exp\left(-\frac{1}{2}\mathbf{R}^{{T \partial M}}-u^{2}\mathbf{S}_{\varsigma_{\mathsf{in}}}^{2}\right)\sum_{k=1}^{\infty}\frac{\left(u\mathbf{S}_{\varsigma_{\mathsf{in}}}\right)^{k}}{2\Gamma(\frac{k}{2}+1)}.

\end{array}
\end{equation}

  \begin{lem}
\label{Lemma_dual_Bruning_Mas_formulas}
$$
\begin{array}{rcl}
\mathbf{e}_{\mathbf{b},\varsigma}({\partial M}, \nabla^{ {T M}})&=&(-1)^{m-1}
\mathbf{e}_{\mathbf{b},-\varsigma}({\partial M}, \nabla^{ {T M}})\\
B_{\varsigma}({\partial M}, \nabla^{ {T M}}) & =
&(-1)^{m-1}B_{-\varsigma}({\partial M}, \nabla^{ {T M}}).
\end{array}
$$
\end{lem}
\begin{proof}
First, note that
$
\mathbf{S}_{\varsigma}=-\mathbf{S}_{-\varsigma}.
$
We compute  $\mathbf{e}_{\mathbf{b},\varsigma}({\partial M}, \nabla^{ {T M}})$ by recalling that
Berezin integrals \emph{see} top degrees terms only:
$$
\begin{array}{rcl}
\mathbf{e}_{\mathbf{b},\varsigma}({\partial M}, \nabla^{ {T M}})& = & (-1)^{m-1}\int^{B_{\partial M}}\exp\left(-\frac{1}{2}(\mathbf{R}^{ {T M}}|_{\partial  M})\right)\sum_{k=0}^{\infty}\frac{\mathbf{S}_{\varsigma}^{k}}{2\Gamma(\frac{k}{2}+1)},\\
& = & (-1)^{m-1}\int^{B_{\partial M}}\sum_{l=0}^{\infty}\frac{-\frac{1}{2}(\mathbf{R}^{ {T M}}|_{\partial  M})^{l}}{l!}\sum_{k=0}^{\infty}\frac{(-1)^{k}\mathbf{S}_{-\varsigma}^{k}}{2\Gamma(\frac{k}{2}+1)},
\\
& = & (-1)^{m-1}\int^{B_{\partial M}}\sum_{l,k=0}^{\infty}\frac{-\frac{1}{2}(\mathbf{R}^{ {T M}}|_{\partial  M})^{l}}{l!}\frac{(-1)^{k}\mathbf{S}_{-\varsigma}^{k}}{2\Gamma(\frac{k}{2}+1)},
\\
& = & (-1)^{m-1}\int^{B_{\partial M}}\sum_{l=0}^{\infty}\frac{-\frac{1}{2}(\mathbf{R}^{ {T M}}|_{\partial  M})^{l}}{l!}
\frac{(-1)^{m-(2l+1)}\mathbf{S}_{-\varsigma}^{m-(2l+1)}}{2\Gamma(\frac{m-(2l+1)}{2}+1)},
\\
& = & (-1)^{m-1}\int^{B_{\partial M}}\sum_{l=0}^{\infty}\frac{-\frac{1}{2}(\mathbf{R}^{ {T M}}|_{\partial  M})^{l}}{l!}
\frac{(-1)^{m-1}\mathbf{S}_{-\varsigma}^{m-(2l+1)}}{2\Gamma(\frac{m-(2l+1)}{2}+1)},
\\
& = & \int^{B_{\partial M}}\sum_{l=0}^{\infty}\frac{-\frac{1}{2}(\mathbf{R}^{ {T M}}|_{\partial  M})^{l}}{l!}
\frac{\mathbf{S}_{-\varsigma}^{m-(2l+1)}}{2\Gamma(\frac{m-(2l+1)}{2}+1)},
\\
& = & \int^{B_{\partial M}}\sum_{l,k=0}^{\infty}\frac{-\frac{1}{2}(\mathbf{R}^{ {T M}}|_{\partial  M})^{l}}{l!}\frac{\mathbf{S}_{-\varsigma}^{k}}{2\Gamma(\frac{k}{2}+1)},
\\
& = &(-1)^{m-1}\mathbf{e}_{\mathbf{b},-\varsigma}({\partial M}, \nabla^{ {T M}}) \\
\end{array}
$$
and analogously for
 the forms $B_{\mp\varsigma}({\partial M}, \nabla^{ {T M}})$.
\end{proof}

\begin{defi}
\label{definitions_of_modified_e_partial_B_forms_on_bordisms}
Define the functions
$\Pi_{\pm}:\partial M\rightarrow\R$ respectively by
\begin{equation}
\label{definitions_Pi_functions}
\begin{array}{lcr}
\Pi_{+}(y):=
\left\{
\begin{array}{cc}
1&\text{if }y\in \partial_{+}M\\
0&\text{if }y\in \partial_{-}M
\end{array}
\right.&\text{and}&
\Pi_{-}(y):=
\left\{
\begin{array}{cc}
0&\text{if }y\in \partial_{+}M\\
1&\text{if }y\in \partial_{-}M.
\end{array}
\right.
\end{array}
\end{equation}
and set
$$
\begin{array}{rl}
\scriptstyle\mathbf{e}_{\partial}( {\partial_{+}M}, {\partial_{-}M}, \nabla^{ {T M}}) & \scriptstyle:=
 i_{+}^{*}\left(\mathbf{e}_{\mathbf{b},\varsigma}( {\partial M}, \nabla^{ {T M}})\right)\Pi_{+1} - i_{-}^{*}\left(\mathbf{e}_{\mathbf{b},\varsigma}( {\partial M}, \nabla^{ {T M}})\right)\Pi_{-1}\\
\\
\scriptstyle\mathbf{e}_{\partial}( {\partial_{-}M}, {\partial_{+}M}, \nabla^{ {T M}}) & \scriptstyle:=
 i_{-}^{*}\left(\mathbf{e}_{\mathbf{b},-\varsigma}( {\partial M}, \nabla^{ {T M}})\right)\Pi_{-1} - i_{+}^{*}\left(\mathbf{e}_{\mathbf{b},-\varsigma}( {\partial M}, \nabla^{ {T M}})\right)\Pi_{+1}\\
\\
\scriptstyle B( {\partial_{+}M}, {\partial_{-}M}, \nabla^{ {T M}}) \scriptstyle& :=
B_{\varsigma}( {\partial M}, \nabla^{ {T M}})\\
\\
\scriptstyle B( {\partial_{-}M}, {\partial_{+}M}, \nabla^{ {T M}}) & \scriptstyle:=
B_{-\varsigma}( {\partial M}, \nabla^{ {T M}})
\
\end{array}
$$
\end{defi}

\begin{lem}
\label{Lemma_duality_e_partial_on_bordisms}
For the forms given in Definition \ref{definitions_of_modified_e_partial_B_forms_on_bordisms}, the relations
$$
\begin{array}{rcl}
\mathbf{e}_{\partial}( {\partial_{+}M}, {\partial_{-}M}, \nabla^{ {T M}})&=&
(-1)^{m}\mathbf{e}_{\partial}( {\partial_{-}M}, {\partial_{+}M}, \nabla^{ {T M}})\\
B( {\partial_{+}M}, {\partial_{-}M}, \nabla^{ {T M}})&=&
(-1)^{m-1}B( {\partial_{-}M}, {\partial_{+}M}, \nabla^{ {T M}})
\end{array}
$$
hold.
\end{lem}
\begin{proof}
This is clear from construction and Lemma \ref{Lemma_dual_Bruning_Mas_formulas}.
\end{proof}

\subsection{Secondary characteristic forms}
Let $\{g_{s}:=g_{s}^{TM}\}_{s\in\R}$ (resp. $\{g_{s}^{\partial}:=g_{s}^{T\partial M}\}_{s\in\R}$) be a smooth family
of Riemannian metrics on $TM$ (resp. the induced family of metrics on $T\partial M$). We sketch the construction in \cite{Bruening-Ma} (see also \cite[(4.53)]{Bismut-Zhang})
for the (secondary) \textit{Chern--Simons forms}  $\widetilde{\mathbf{e}}\left(M,g_{0},g_{s}\right)$
and $\left.\widetilde{\mathbf{e}}_{\mathbf{b}}\right.\left( \partial M,g_{0},g_{s}\right)$.

Let
$\nabla_{s}:=\nabla_{g_{s}}^{TM}$ and $\mathsf{R}_{s}:={\mathsf{R}}^{TM}_{g_{s}}$ (resp.
$\nabla_{s}^{\partial}:=\nabla_{g^{\partial}_{s}}^{T\partial M}$ and $\mathsf{R}^{\partial}_{s}:=\mathsf{R}^{T\partial M}_{g^{\partial}_{s}}$)
be the Levi-Civit\`a connections and curvatures
on $TM$ (resp. on $T\partial M$) associated to the metrics $g_{s}$ (resp. $g^{\partial}_{s}$).
Consider the \textit{deformation spaces} $\widetilde{M}:=M\times\R$ (resp. $\widetilde{\partial M}:=\partial M\times \R$) with
$\pi_{\widetilde{M}}:\widetilde{M}\rightarrow \R\text{ and }\mathbf{p}_{M}:\widetilde{M}\rightarrow M,$ its canonical projections (resp.
$\pi_{\widetilde{\partial M}}:\widetilde{\partial M}\rightarrow \R\text{ and }\mathbf{p}_{\partial M}:
\widetilde{\partial M}\rightarrow \partial M$).
If $\widetilde{i}:=i\times\mathbf{id_\R}:\widetilde{\partial M}\rightarrow \widetilde{M}$ is the natural embedding
induced by $i:\partial M\rightarrow M$,
then
$\pi_{\widetilde{\partial M}}=\pi_{\widetilde{M}}\circ \widetilde{i}$.
The \textit{vertical bundle} of the fibration $\pi_{\widetilde{M}}:\widetilde{M}\rightarrow \R$ (resp. $\pi_{\widetilde{\partial M}}:\widetilde{\partial M}\rightarrow\R$)
is the pull-back of the tangent bundle $TM\rightarrow M$ along  $\mathbf{p}_{M}:\widetilde{M}\rightarrow M$
(resp. the pull-back of $T\partial M\rightarrow \partial M$ along  $\mathbf{p}_{\partial M}:\widetilde{\partial M}\rightarrow \partial M$), i.e.,
\begin{equation}
 \label{equation_mathcal_T_M}
 \mathcal{T M}:=\mathbf{p}_{M}^{*}TM\rightarrow \widetilde{M},\quad(\text{resp. }\mathcal{T\partial M}:=\mathbf{p}_{\partial M}^{*}T\partial M\rightarrow \widetilde{\partial M})
\end{equation}
and it is considered as a subbundle of $T\widetilde{M}$ (resp. $T\widetilde{\partial M}$).
The bundle $\mathcal{T M}$ (resp. $\mathcal{T\partial M}$) in (\ref{equation_mathcal_T_M}) is naturally equipped with a Riemannian metric $g^{\mathcal{T M}}$ which
coincides with $g_{s}$ (resp. $g^{\partial}_{s}$) at $M\times\{s\}$ (resp. $\partial M\times\{s\}$), for which
there exists a unique natural metric connection $\nabla^{\mathcal{T M}}$ (resp. $\nabla^{\mathcal{T \partial M}}$) whose
curvature tensor is denoted by $\mathsf{R}^{\mathcal{T M}}$ (resp. $\mathsf{R}^{\mathcal{T \partial M}}$); for more details,
see \cite[Section 1.5, (1.44) and Definition 1.1]{Bruening-Ma}, and also \cite[(4.50) and (4.51)]{Bismut-Zhang}).
Near the boundary, consider orthonormal frames of $\mathcal{TM}$ such that
$e_{m}(y, s) = \varsigma_{}$ for each $y\in\partial M$ with respect to the metric $g_{s}$. Finally,
by using the
formalism described above associated to $\mathsf{R}^{\mathcal{T M}}$ and $\mathsf{R}^{\mathcal{T \partial M}}$ to define (\ref{label_in_article_differential_forms_boundary}),
if $\mathsf{inc}_{s}: M \rightarrow  \widetilde{M}$
is the inclusion map given by $\mathsf{inc}_{s}(x)=(x,s)$ for $x_{0}\in M$ and $s\in \R$, then one defines
\begin{defi}
  \label{definition_of_Chern_Simons_secondary_classes_1}
  $$
  \begin{array}{rcll}
  \widetilde{\mathbf{e}}\left(M,g_{0},g_{\tau}\right)&:=&
  \int_{0}^{\tau}\mathsf{inc}^{*}_{s}
  \left(\iota\left(\frac{\partial}{\partial s}\right)
  \mathbf{e}(\widetilde{M},\nabla^{\mathcal{T M}})\right)ds\\
  &&\quad\quad\quad\quad\quad\quad\quad\quad\quad\in\quad\Omega^{m-1}(M,\Theta_{M})\\
 \left.\widetilde{\mathbf{e}}_{\partial}\right.\left(\partial_{+}M, \partial_{-}M,g_{0},g_{\tau}\right)&:=&
  \int_{0}^{\tau}\mathsf{inc}^{*}_{s}
  \left(\iota\left(\frac{\partial}{\partial s}\right)
  \left.\mathbf{e}_{\partial}\right.(\widetilde{\partial_{+}M}, \widetilde{\partial_{-}M},\nabla^{\mathcal{T M}})\right) ds\\
  &&\quad\quad\quad\quad\quad\quad\quad\quad\quad\in\quad\Omega^{m-2}(\partial M,\Theta_{M})\\
  \left.\widetilde{\mathbf{e}}_{\mathbf{b}}\right.\left( \partial M,g_{0},g_{\tau}\right)&:=&
  \int_{0}^{\tau}\mathsf{inc}^{*}_{s}
  \left(\iota\left(\frac{\partial}{\partial s}\right)
  \left.\mathbf{e}_{\mathbf{b}}\right.(\widetilde{\partial M}, \nabla^{\mathcal{T M}})\right) ds\\
  &&\quad\quad\quad\quad\quad\quad\quad\quad\quad\in\quad\Omega^{m-2}(\partial M,\Theta_{M})
  \end{array}
  $$
\end{defi}
where $\iota(X)$ indicates the contraction with respect to the vector field $X$.

\end{document}